\documentclass[10pt,reqno]{amsart}
\usepackage[margin=1in]{geometry}
\usepackage{amssymb,mathrsfs,color}
\usepackage{enumerate}
\usepackage{graphicx}
\usepackage{xcolor} 
\usepackage{geometry}
\usepackage{amsfonts,amssymb,amsmath}
\usepackage{slashed}
\usepackage{ mathrsfs }
\usepackage[titletoc,title]{appendix}
\usepackage{wrapfig}

\usepackage{cancel}

\usepackage{cite}

\usepackage{nth}

\usepackage{marginnote}

\usepackage{pgfplots}

\usepackage{verbatim}

\definecolor{green}{rgb}{0,0.8,0} 

\definecolor{deepgreen}{cmyk}{1,0,1,0.5}

\newcommand{\Del}[1]{}

\numberwithin{equation}{section}

\newtheorem{theorem}{Theorem}[section]
\newtheorem{corollary}[theorem]{Corollary}
\newtheorem{lemma}[theorem]{Lemma}
\newtheorem{proposition}[theorem]{Proposition}
\newtheorem{remark}[theorem]{Remark}

\newcommand{\jap}[1]{\langle #1\rangle}

\renewcommand{\Im}{\mathrm{Im}}

\renewcommand{\hbar}{{\underline h}}

\newcommand{\bbN}{\mathbb N}

\newcommand{\bbR}{\mathbb R}

\newcommand{\bbZ}{\mathbb Z}

\newcommand{\calB}{\mathcal B}

\newcommand{\calD}{\mathcal D}
\newcommand{\calE}{\mathcal E}
\newcommand{\calF}{\mathcal F}

\newcommand{\calO}{\mathcal O}

\newcommand{\calR}{\mathcal R}
\newcommand{\calS}{\mathcal S}

\newcommand{\bR}{{\mathbb R}}

\newcommand{\ud}{\mathrm{d}}

\newcommand{\px}{\partial_x}
\newcommand{\pt}{\partial_t}

\DeclareMathOperator{\sech}{sech}

\begin{document}

\title[Decay and asymptotics for the 1D variable coefficient cubic NLKG]{Decay and asymptotics for the 1D Klein-Gordon equation with variable coefficient cubic nonlinearities}

\author[H. Lindblad]{Hans Lindblad}
\address{Department of Mathematics \\ Johns Hopkins University \\ Baltimore, MD 21218, USA}
\email{lindblad@math.jhu.edu}

\author[J. L\"uhrmann]{Jonas L\"uhrmann}
\address{Department of Mathematics \\ Texas A\&M University \\ College Station, TX 77843, USA}
\email{luhrmann@math.tamu.edu}

\author[A. Soffer]{Avy Soffer}
\address{Mathematics Department, Rutgers University, New Brunswick, NJ 08903, USA}
\email{soffer@math.rutgers.edu}

\thanks{H. Lindblad is supported in part by NSF grant DMS-1500925. A. Soffer is supported in part by NSF grant DMS-1600749 and NSFC11671163. Part of this work was conducted while the last two authors were visiting CCNU, Wuhan, China.
}

\begin{abstract}
 We obtain sharp decay estimates and asymptotics for small solutions to the one-dimensional Klein-Gordon equation with constant coefficient cubic and spatially localized, variable coefficient cubic nonlinearities. Vector-field techniques to deal with the long-range nature of the cubic nonlinearity become problematic in the presence of variable coefficients. We introduce a novel approach based on pointwise-in-time local decay estimates for the Klein-Gordon propagator to overcome this impasse.
\end{abstract}

\maketitle

\section{Introduction}

We consider the one-dimensional Klein-Gordon equation with constant and variable coefficient cubic nonlinearities
\begin{equation} \label{equ:nlkg}
 \left\{ \begin{aligned}
  (\partial_t^2 - \partial_x^2 + 1) u &= \beta_0 u^3 + \beta(x) u^3 \text{ on } \bbR^{1+1}, \\
  (u, \partial_t u)|_{t=1} &= (f,g),
 \end{aligned} \right.
\end{equation}
where $\beta_0 \in \bbR$ and $\beta(x)$ is a real-valued Schwartz function. For small, smooth initial data the global existence of solutions to~\eqref{equ:nlkg} follows readily from energy conservation. The goal of this paper is to establish dispersive decay of such solutions and to uncover their asymptotics. For technical convenience the initial data is prescribed at time $t = 1$ and we assume that $(f,g)$ are real-valued, smooth, and decaying. In particular, we do not make any compact support assumptions.

\medskip

Our interest in this problem is motivated by the asymptotic stability analysis of topological solitons arising in some equations of mathematical physics. A well-known example in one space dimension is the $\phi^4$ model
\begin{equation*}
 (\partial_t^2 - \partial_x^2) \phi = \phi - \phi^3,
\end{equation*}
which admits the stationary ``kink'' solution
\begin{equation*}
 \phi_K(x) = \tanh( {\textstyle \frac{x}{\sqrt{2}} } ).
\end{equation*}
The linearization $\phi = \phi_K + u$ about the kink gives rise to the following one-dimensional Klein-Gordon equation for the perturbation
\begin{equation} \label{equ:linearization_kink}
 \bigl( \partial_t^2 - \partial_x^2 + 2 - 3 \sech^2( {\textstyle \frac{x}{\sqrt{2}} }) \bigr) u = - 3 \tanh( {\textstyle \frac{x}{\sqrt{2}} } ) u^2 - u^3,
\end{equation}
whose nonlinearity consists of a \emph{variable coefficient} quadratic and a constant coefficient cubic term. The asymptotic stability analysis of the kink $\phi_K$ therefore requires to understand the convergence to zero of any small solution to nonlinear Klein-Gordon equations of the type~\eqref{equ:linearization_kink}, see~\cite{S06}. This problem comes with several major challenges. Due to the slow decay of Klein-Gordon waves in one space dimension the quadratic and cubic nonlinearities on the right-hand side of~\eqref{equ:linearization_kink} exhibit long-range effects that are significantly compounded and altered by the presence of variable coefficients. Additionally, the linearized operator on the left-hand side of~\eqref{equ:linearization_kink} possesses an internal mode of oscillation causing a discrepancy between the decay rates of different components of the solution.

The remarkable work of Kowalczyk-Martel-Mu\~{n}oz~\cite{KMM17} established the asymptotic stability of the kink with respect to a local energy norm under small, odd, finite energy perturbations, see also the review~\cite{KMM17_1} and references therein. For general perturbations the asymptotic stability of the kink remains an outstanding open problem. Moreover, \cite{KMM17} does not address the question of determining the precise asymptotic behavior of small perturbations, possibly with respect to a stronger topology. A key step in this direction is to uncover the dispersive decay and the asymptotics of small solutions to one-dimensional Klein-Gordon equations without a linear potential, but with constant as well as variable coefficient quadratic and cubic nonlinearities of the type
\begin{equation} \label{equ:model_kink}
 (\partial_t^2 - \partial_x^2 + 1) u = (\alpha_0 + \alpha(x)) u^2 + (\beta_0 + \beta(x)) u^3 \text{ on } \bbR^{1+1}.
\end{equation}
While~\eqref{equ:linearization_kink} only features a variable coefficient quadratic term, we note that variable coefficient cubic nonlinearities occur in the equations for perturbations of kinks in many other 1D nonlinear scalar field theories, for instance in the well-known sine-Gordon equation.

\medskip 

The study of the long-time behavior of small solutions to Klein-Gordon equations with constant coefficient nonlinearities has a long history, starting with the pioneering works of Klainerman~\cite{Kl80, Kl85} and Shatah~\cite{Sh85}. The rich and vast literature on this subject cannot be reviewed here in its entirety and our focus is on those works that are most relevant to this paper. 

In one space dimension the slow decay of free Klein-Gordon waves causes quadratic and cubic nonlinearities to exhibit long-range effects. One therefore cannot expect the respective nonlinear solutions to have the same asymptotics as linear solutions. Indeed, the seminal work of Delort~\cite{Del01, Del06} showed that the asymptotic behavior of small solutions to a one-dimensional Klein-Gordon equation with quadratic and cubic nonlinearities differs from the behavior of free Klein-Gordon waves by a logarithmic phase correction. Inspired by~\cite{Lind90, Lind92, LR03}, a different approach was later developed by Lindblad-Soffer~\cite{LS05_1, LS05_2} for the cubic case, providing a detailed asymptotic expansion of the solution for large times. Subsequently, Hayashi-Naumkin~\cite{HN08, HN12} removed the compact support assumptions about the initial data required in~\cite{Del01, LS05_1, LS05_2}, see also Stingo~\cite{Stingo18} and Candy-Lindblad~\cite{CL18}.

For related results on the asymptotic behavior of small solutions to the one-dimensional cubic nonlinear Schr\"odinger equation, we refer the reader to Hayashi-Naumkin~\cite{HN98}, Lindblad-Soffer~\cite{LS06}, Kato-Pusateri~\cite{KatPus11}, Ifrim-Tataru~\cite{IT15}, and references therein. 


In contrast, the study of one-dimensional nonlinear Klein-Gordon equations with variable coefficient nonlinearities has only recently been initiated by Lindblad-Soffer~\cite{LS15} and Sterbenz~\cite{Sterb16}. Specifically, \cite{LS15, Sterb16} establish dispersive decay of small solutions for smooth, compactly supported initial data in the case of a nonlinearity consisting of a spatially localized, variable coefficient cubic term coupled to constant coefficient quadratic and cubic terms, i.e. in the case of~\eqref{equ:model_kink} with $\alpha_0, \beta_0, \beta(x) \neq 0$ and $\alpha(x) = 0$. 

The presence of a variable coefficient nonlinearity poses significant challenges to current techniques of dealing with the long-range effects of (non-localized) constant coefficient quadratic and cubic nonlinearities in one space dimension. Oversimplifying a bit here, one generally uses a combination of vector-field and normal form methods together with an ODE argument that uncovers the asymptotic behavior of the solution. In the case of the Klein-Gordon equation, the only weighted vector-field that commutes with the linear flow is the Lorentz boost $Z = t \partial_x + x \partial_t$. However, differentiation of an ($x$-dependent) variable coefficient by Lorentz boosts produces strongly divergent factors of $t$, which seem to put suitable energy estimates out of reach. In~\cite{LS15, Sterb16} the main idea to overcome this issue is the introduction of a novel variable coefficient cubic normal form. It recasts the variable coefficient nonlinearity into a better form that allows one to obtain slowly growing energy estimates for one Lorentz boost of the solution.

Finally, we mention the recent works of Delort~\cite{Del16} and Chen-Pusateri~\cite{ChenPus19} on modified scattering for solutions to 1D Schr\"odinger equations with a linear potential and with constant as well as variable coefficient cubic nonlinearities, see also Germain-Pusateri-Rousset~\cite[p. 1479]{GermPusRou18}. In the case without a linear potential, which is the Schr\"odinger analogue of the Klein-Gordon equation~\eqref{equ:nlkg} considered here, \cite{Del16, ChenPus19} obtain decay and asymptotics of small solutions in the special case of odd initial data and even decaying variable coefficients.

\medskip 

In this work we continue the analysis of~\cite{LS15, Sterb16} in the case of the one-dimensional Klein-Gordon equation with constant as well as spatially localized, variable coefficient cubic nonlinearities, i.e. in the case of~\eqref{equ:model_kink} with $\beta_0, \beta(x) \neq 0$ and $\alpha_0 = \alpha(x) = 0$. We are now in the position to state our main results.
\begin{theorem} \label{thm:main}
 Let $N \geq 2$. There exists $\varepsilon_0 > 0$ such that if the initial data satisfy
 \begin{equation*}
  \bigl\| \jap{x}^{1+\frac{N}{2}} f \bigr\|_{H^{N+2}_x(\bbR)} + \bigl\| \jap{x}^{1+\frac{N}{2}} g \bigr\|_{H^{N+1}_x(\bbR)} \leq \varepsilon_0,
 \end{equation*}
 then the global solution to~\eqref{equ:nlkg} satisfies the following decay estimate in the exterior region $1 \leq t \leq \jap{x}$,
 \begin{equation} \label{equ:main_exterior_estimate}
  |u(t,x)| \lesssim \jap{x}^{-\frac{N}{2}} \Bigl( \bigl\| \jap{x}^{1+\frac{N}{2}} f \bigr\|_{H^{N+2}_x(\bbR)} + \bigl\| \jap{x}^{1+\frac{N}{2}} g \bigr\|_{H^{N+1}_x(\bbR)} \Bigr),
 \end{equation}
 while in the interior region $\jap{x} \leq t$ it holds that
 \begin{equation} \label{equ:main_interior_estimate}
  |u(t,x)| \lesssim \frac{1}{\sqrt{t}}  \Bigl( \bigl\| \jap{x}^{1+\frac{N}{2}} f \bigr\|_{H^{N+2}_x(\bbR)} + \bigl\| \jap{x}^{1+\frac{N}{2}} g \bigr\|_{H^{N+1}_x(\bbR)} \Bigr).
 \end{equation}
\end{theorem}

Theorem~\ref{thm:main} improves over the prior results~\cite{LS15, Sterb16} for~\eqref{equ:nlkg} by establishing a sharp decay estimate in terms of the time variable $t$ instead of the variable $\rho$ that defines a foliation of the interior of the light cone by hyperboloids $H_\rho := \{ (t,x) \in \bbR^{1+1} : t^2 - x^2 = \rho^2 \}$. Moreover, we do not require the initial data to be compactly supported and we go beyond the analysis in~\cite{LS15, Sterb16} by also uncovering the asymptotics of the solution to~\eqref{equ:nlkg}, see Corollary~\ref{cor:asymptotics} below. 

What we view as the main novelty of this work is the introduction of a new and robust method to deal with the difficulties of deriving slowly growing energy estimates of a Lorentz boost of the solution in the presence of a variable coefficient cubic nonlinearity. While~\cite{LS15, Sterb16} introduced a delicate type of variable coefficient cubic normal form for this purpose, our approach here avoids any use of normal forms and is instead based on pointwise-in-time local decay estimates for the free Klein-Gordon propagator of the type
\begin{equation} \label{equ:local_decay_example}
 \bigl\| \jap{x}^{-\sigma} e^{\pm i t \jap{\nabla}} \jap{x}^{-\sigma} \bigr\|_{L^2_x(\bbR) \to L^2_x(\bbR)} \lesssim \frac{1}{\jap{t}^{\frac{1}{2}}}, \quad \sigma \geq 1.
\end{equation}
See Subsection~\ref{subsec:main_proof_ideas} below for an outline of the main proof ideas. We expect that our method will allow for more progress in the variable coefficient quadratic case ($\alpha(x) \neq 0$) in~\eqref{equ:model_kink} and lead to a complete understanding of the long-range effects of the quadratic and cubic nonlinearities in~\eqref{equ:linearization_kink} in the presence of variable coefficients.
Pointwise-in-time local decay estimates of the form~\eqref{equ:local_decay_example}, or propagation estimates, for much larger classes of unitary operators $e^{-itH}$ originate in the works of Rauch~\cite{Rauch78}, Jensen-Kato~\cite{KJ79}, and Jensen~\cite{Jensen80, Jensen84} in the context of Schr\"odinger operators $H = -\Delta + V(x)$, see also~\cite{HSS99, JSS91, Ger08, GLS16, LarS15, Schl07} and references therein.

\begin{remark}
Our approach can easily be extended to also include a constant coefficient quadratic nonlinearity ($\alpha_0 \neq 0$) by incorporating a version of Shatah's normal form method~\cite{Sh85}. We do not pursue this here in order not to obfuscate the main novelty of our argument, namely the use of pointwise-in-time local decay estimates for the Klein-Gordon propagator in the derivation of energy estimates of a Lorentz boost of the solution in the presence of a spatially localized, variable coefficient cubic nonlinearity.
\end{remark}

\medskip 

The proof of Theorem~\ref{thm:main} also uncovers the asymptotic behavior of small solutions to~\eqref{equ:nlkg} in the interior region $\jap{x} \leq t$. It can be conveniently expressed in terms of hyperbolic coordinates $(\rho, y) \in [1, \infty) \times \bbR$ determined by $t = \rho \cosh(y)$ and $x = \rho \sinh(y)$.

\begin{corollary} \label{cor:asymptotics}
 Suppose that the assumptions of Theorem~\ref{thm:main} are in place. Then there exist a function $a \in L^\infty_y(\bbR) \cap L^2_y(\bbR)$ and a small constant $0 < \nu \ll 1$ such that in the interior region $\jap{x} \leq t$ the asymptotic behavior of the solution to~\eqref{equ:nlkg} is given by
 \begin{equation} \label{equ:main_asymptotics}
  u(t,x) = \frac{1}{\sqrt{t}} \, \Im \Bigl( e^{i ( \rho - \frac{3}{8} \beta_0 \frac{|a(y)|^2}{\cosh(y)} \log(\rho) )} a(y) \Bigr) + \frac{1}{\sqrt{t}} \calO( \rho^{-\nu})
 \end{equation}
 as $\rho = \sqrt{t^2-x^2} \to \infty$.
\end{corollary}

\begin{remark}
While the constant coefficient cubic nonlinearity causes a logarithmic phase correction, the spatially localized, variable coefficient cubic nonlinearity in~\eqref{equ:nlkg} does not lead to a phase correction (in terms of $\rho$). However, we expect that the regularity and decay properties of the amplitude $a(y)$ are affected by the variable coefficient, which we plan to consider in a future investigation. 
\end{remark}

\subsection{Hyperbolic coordinates and equations}

 We foliate the interior region $\jap{x} \leq t$ in terms of hyperboloids
 \begin{equation*}
  H_\rho := \bigl\{ (t,x) \in [0,\infty) \times \bbR \, : \, \rho^2 = t^2 - x^2 \bigr\}, \quad \rho \geq 1,
 \end{equation*}
 and use hyperbolic coordinates $(\rho, y) \in [1, \infty) \times \bbR$ determined by
 \begin{equation} \label{equ:definition_hyperbolic_coordinates}
  x = \rho \sinh(y), \quad t = \rho \cosh(y).
 \end{equation}
 Then it holds that
 \begin{equation*}
  \partial_\rho = \rho^{-1} (t \partial_t + x \partial_x), \quad \partial_y = t \partial_x + x \partial_t,
 \end{equation*}
 and the linear Klein-Gordon operator can be written as
 \begin{equation*}
  \Box + 1 = \partial_t^2 - \partial_x^2 + 1 = \partial_\rho^2 + \frac{1}{\rho} \partial_\rho - \frac{1}{\rho^2} \partial_y^2 + 1.
 \end{equation*}
 Conjugating $\Box + 1$ by $t^{\frac{1}{2}}$ we find that
 \begin{equation} \label{equ:conjugate_by_thalf}
  \begin{aligned}
   t^{\frac{1}{2}} (\Box + 1) = t^{\frac{1}{2}} (\partial_t^2 - \partial_x^2 + 1) 
   = \Bigl( \pt^2 - \frac{1}{t} \pt - \px^2 + 1 + \frac{3}{4 t^2} \Bigr) t^{\frac{1}{2}}.
  \end{aligned}
 \end{equation}
 In particular, in hyperbolic coordinates the identity~\eqref{equ:conjugate_by_thalf} reads 
 \begin{equation*}
  \rho^{\frac{1}{2}} \cosh^{\frac{1}{2}}(y) (\Box + 1) = \Bigl( \partial_\rho^2 - \frac{1}{\rho^2} \partial_y^2 + \frac{1}{\rho^2} \tanh(y) \partial_y + 1 + \frac{3}{4 \rho^2} \cosh^{-2}(y) \Bigr) \rho^{\frac{1}{2}} \cosh^{\frac{1}{2}}(y).
 \end{equation*}
 Hence, upon introducing the variable
 \begin{equation*}
  w = t^{\frac{1}{2}} u = \rho^{\frac{1}{2}} \cosh^{\frac{1}{2}}(y) u,
 \end{equation*}
 we may write the nonlinear Klein-Gordon equation~\eqref{equ:nlkg} as 
 \begin{equation} \label{equ:nlkg_hyperbolic_coord}
  \begin{aligned}
   \Bigl( \partial_\rho^2 - \frac{1}{\rho^2} \partial_y^2 + \frac{1}{\rho^2} \tanh(y) \partial_y + 1 + \frac{3}{4 \rho^2} \cosh^{-2}(y) \Bigr) w = \frac{1}{\rho} \frac{\beta_0}{\cosh(y)} w^3 + \frac{1}{\rho} \frac{\beta(\rho \sinh(y))}{\cosh(y)} w^3.
  \end{aligned}
 \end{equation}

\subsection{Outline of the proofs of Theorem~\ref{thm:main} and Corollary~\ref{cor:asymptotics}} \label{subsec:main_proof_ideas}

In the proof of Theorem~\ref{thm:main} we treat the interior and the exterior region separately.
The derivation of the exterior decay estimate~\eqref{equ:main_exterior_estimate} only relies on the stronger decay of the Klein-Gordon equation in the exterior of the light cone. We use an energy methods argument by Klainerman~\cite{Kl93}, following the presentation in Candy-Lindblad~\cite{CL18}.

For the proof of the interior decay estimate~\eqref{equ:main_interior_estimate} we foliate the interior region in terms of hyperboloids~$H_\rho$ and mostly work in hyperbolic coordinates. Then~\eqref{equ:main_interior_estimate} is equivalent to a uniform bound $\|w\|_{L^\infty_\rho L^\infty_y} \lesssim~\varepsilon$ for the variable $w$. By one-dimensional Sobolev embedding this would follow from a uniform energy bound $\|w\|_{L^\infty_\rho H^1_y} \lesssim \varepsilon$. The contribution of the constant coefficient cubic nonlinearity to such an $L^\infty_\rho H^1_y$ energy estimate is of the form $\int_1^R \rho^{-1} \|w(\rho)\|_{L^\infty_y}^2 \|w(\rho)\|_{H^1_y} \, \ud \rho$, which just about fails to be integrable. Energy estimates alone will therefore not suffice and we can at most hope to obtain a slowly growing energy estimate 
\begin{equation} \label{equ:intro_slow_H1y}
 \|w(\rho)\|_{H^1_y(\bbR)} \lesssim \varepsilon \rho^{\delta}
\end{equation}
for some small $0 < \delta \ll 1$ within a suitable bootstrap argument on the $L^\infty_\rho L^\infty_y$-norm of $w$. 
However, the contribution of the variable coefficient cubic nonlinearity seems to actually place a slowly growing energy estimate of the form~\eqref{equ:intro_slow_H1y} far out of reach, since the Lorentz boost $\partial_y$ produces a strongly divergent factor of~$\rho$ when it falls onto the variable coefficient $\beta( \rho \sinh(y) )$. Nevertheless, it turns out that one can judiciously combine pointwise-in-time local decay estimates for the Klein-Gordon propagator with the spatial localization of the variable coefficient $\beta(x)$ to counteract such a divergent factor and in this manner actually salvage a slowly growing energy bound of the form~\eqref{equ:intro_slow_H1y}.

Before turning to this, we outline how a slowly growing energy estimate of the form~\eqref{equ:intro_slow_H1y} implies a uniform $L^\infty_\rho L^\infty_y$ bound for $w$, and thus allows to close a bootstrap argument to infer the interior decay estimate~\eqref{equ:main_interior_estimate}.
This part of our proof uses many ideas from~\cite{LS05_1, LS05_2, LS15, Sterb16}.
The starting point is to decompose the variable $w = P_{\leq \rho^\sigma} w + P_{> \rho^\sigma} w$, using a $\rho$-dependent frequency cut-off for some small $0 < \delta \ll \sigma \ll 1$. Then the uniform boundedness of the high-frequency component $P_{> \rho^\sigma} w$ follows easily from~\eqref{equ:intro_slow_H1y} by exploiting that control of the $H^{\frac{1}{2}+\varepsilon}_y$-norm of $w$ suffices to bound its $L^\infty_y$-norm.
The low-frequency component $P_{\leq \rho^\sigma} w$ satisfies the evolution equation
\begin{equation} \label{equ:intro_equation_low_freq}
 \Bigl( \partial_\rho^2  - \frac{1}{\rho^2} \partial_y^2 + 1 \Bigr) (P_{\leq \rho^\sigma} w) = \frac{1}{\rho} P_{\leq 2^{10} \rho^\sigma} \biggl( \frac{\beta_0 + \beta(\rho \sinh(y))}{\cosh(y)} \biggr) (P_{\leq \rho^\sigma} w)^3 + \ldots,
\end{equation}
where we omit various remainder terms on the right-hand side that are integrable thanks to~\eqref{equ:intro_slow_H1y}. Moreover, thanks to the slowly growing energy bound~\eqref{equ:intro_slow_H1y} we can also treat $\rho^{-2} \partial_y^2 (P_{\leq \rho^\sigma} w)$ as an integrable error term and thus view~\eqref{equ:intro_equation_low_freq} as an ODE 
\begin{equation} \label{equ:intro_ODE_low_freq}
 (\partial_\rho^2 + 1) (P_{\leq \rho^\sigma} w) = \frac{1}{\rho} P_{\leq 2^{10} \rho^\sigma} \biggl( \frac{\beta_0 + \beta(\rho \sinh(y))}{\cosh(y)} \biggr) (P_{\leq \rho^\sigma} w)^3 + \ldots
\end{equation}
Now we are in the position to deduce the desired uniform boundedness of the low-frequency component $P_{\leq \rho^\sigma} w$ by invoking an $L^\infty_\rho$-estimate for the ODE $(\partial_\rho^2 + 1) v = \rho^{-1} \tilde{\beta} v^3$ (by multiplying by $\partial_\rho v$) and by treating all other terms in~\eqref{equ:intro_equation_low_freq} as integrable error terms using~\eqref{equ:intro_slow_H1y}. We
remark that similar ideas were also used in~\cite{Lind90, Lind92, LR03} for the wave equation and are related to the weak null condition.

Finally, we discuss the derivation of the slowly growing energy estimate~\eqref{equ:intro_slow_H1y}. To avoid several technicalities and to only focus on the main ideas, here we consider a related, slightly simpler task in rectilinear coordinates $(t,x)$. Namely, to obtain a slowly growing energy estimate for one Lorentz boost of the solution to~\eqref{equ:nlkg} of the type
\begin{equation} \label{equ:intro_slow_Zu}
 \|Zu(t)\|_{L^2_x(\bbR)} \lesssim \varepsilon t^{\delta},
\end{equation}
under a bootstrap assumption $\|u(t)\|_{L^\infty_x(\bbR)} \lesssim \varepsilon \jap{t}^{-\frac{1}{2}}$. This argument can then be upgraded to deduce~\eqref{equ:intro_slow_H1y}, see Subsection~\ref{subsec:interior_energy_estimate} for the details.

In order to better isolate the contribution of the variable coefficient cubic nonlinearity to energy estimates, we decompose the solution to~\eqref{equ:nlkg} into 
\begin{equation} \label{equ:decomposition_solution}
 u(t) = u_0(t) + u_1(t), 
\end{equation}
where the components satisfy
\begin{align*}
 (\Box + 1) u_0 &= \beta_0 u^3 = \beta_0 (u_0 + u_1)^3, &\text{ with } (u_0, \partial_t u_0)|_{t=1} = (f, g), \\
 (\Box + 1) u_1 &= \beta(x) u^3 = \beta(x) (u_0 + u_1)^3,  &\text{ with } (u_1, \partial_t u_1)|_{t=1} = (0,0).
\end{align*}

We recall that the Klein-Gordon energy functional $E(\phi) = \int_{\bbR} (\partial_t \phi)^2 + (\partial_x \phi)^2 + \phi^2 \, \ud x$ satisfies 
\begin{equation*}
 \partial_t E(\phi) = 2 \int_{\bbR} \bigl( (\Box + 1) \phi \bigr) \partial_t \phi \, \ud x.
\end{equation*}
Obviously, if we establish a slow growth bound $\|Z u_\ell(t)\|_{L^2_x(\bbR)} \lesssim \varepsilon t^{\delta}$, $\ell = 0, 1$, for each component separately, then the desired estimate~\eqref{equ:intro_slow_Zu} for the entire solution $u(t)$ follows. Such a bound is straightforward for the component $u_0(t)$. Instead, for the energy of $(Zu_1)(t)$ we find that
\begin{align*}
 \partial_t E(Z u_1) = 2 \int_{\bbR} \bigl( (\Box + 1) Z u_1 \bigr) \partial_t Zu_1 \, \ud x = 2 \int_{\bbR} t \, (\partial_x \beta)(x) u^3 (\partial_t Z u_1) \, \ud x + \ldots,
\end{align*}
where on the right-hand side we only display the worst contribution when $Z$ falls onto the variable coefficient. Due to the divergent factor of $t$, at first sight a slow growth estimate for $E(Z u_1)$ seems out of reach. However, at this point the key idea is to realize that $(\partial_t Z u_1)(t)$ satisfies a weighted energy estimate of the form
\begin{equation} \label{equ:intro_weighted_Zu1}
 \| \jap{x}^{-2} (\partial_t Z u_1)(t) \|_{L^2_x(\bbR)} \lesssim \frac{\varepsilon^3}{\jap{t}^{\frac{1}{2}}},
\end{equation}
which can provide additional decay. Indeed, exploiting the spatial localization of the variable coefficient $\beta(x)$ and using the bootstrap assumption $\|u(t)\|_{L^\infty_x(\bbR)} \lesssim \varepsilon \jap{t}^{-\frac{1}{2}}$, we may bound
\begin{align*}
 | \partial_t E(Z u_1) | &\lesssim t \, \| \jap{x}^2 \partial_x \beta(x) \|_{L^2_x(\bbR)} \|u(t)\|_{L^\infty_x(\bbR)}^3 \| \jap{x}^{-2} \partial_t Zu_1(t) \|_{L^2_x(\bbR)} + \ldots \lesssim \frac{\varepsilon^6}{\jap{t}} + \ldots
\end{align*}
Now the right-hand side only barely fails to be integrable in time! Correspondingly, we can actually obtain a slow growth estimate for the energy $E(Zu_1)$, which implies the desired bound $\|Z u_1(t)\|_{L^2_x(\bbR)} \lesssim \varepsilon^3 t^{\delta}$ for the component $u_1(t)$.

In order to arrive at the weighted energy estimate~\eqref{equ:intro_weighted_Zu1}, we combine the spatial localization provided by the variable coefficient $\beta(x)$ with pointwise-in-time local decay estimates for the Klein-Gordon propagator of the form
\begin{equation*}
 \bigl\| \jap{x}^{-2} e^{\pm i t \jap{\nabla}} \jap{x}^{-2} \bigr\|_{L^2_x(\bbR) \to L^2_x(\bbR)} \lesssim \frac{1}{\jap{t}^{\frac{1}{2}}}, \quad \text{and} \quad \Bigl\| \jap{x}^{-2} \frac{\partial_x}{\jap{\nabla}} e^{\pm i t \jap{\nabla}} \jap{x}^{-2} \Bigr\|_{L^2_x(\bbR) \to L^2_x(\bbR)} \lesssim \frac{1}{\jap{t}^{\frac{3}{2}}}.
\end{equation*}
Then decomposing $\partial_t Z u_1 = \partial_x u_1 + t \partial_x \partial_t u_1 + \ldots$ and writing $u_1(t)$ in Duhamel form, we find that 
\begin{align*}
 \bigl\| \jap{x}^{-2} \partial_x u_1(t) \bigr\|_{L^2_x(\bbR)} &\lesssim \int_1^t \Bigl\| \jap{x}^{-2} \frac{\partial_x}{\jap{\nabla}} \sin((t-s)\jap{\nabla}) \jap{x}^{-2} \Bigr\|_{L^2_x(\bbR) \to L^2_x(\bbR)} \bigl\| \jap{x}^2 \beta(x) \|_{L^2_x(\bbR)} \| u(s) \|_{L^\infty_x(\bbR)}^3 \, \ud s \\
 &\lesssim \int_1^t \frac{1}{\jap{t-s}^{\frac{3}{2}}} \frac{\varepsilon^3}{\jap{s}^{\frac{3}{2}}} \, \ud s \\
 &\lesssim \frac{\varepsilon^3}{\jap{t}^{\frac{3}{2}}}.
\end{align*}
In an analogous manner we may conclude that $\| \jap{x}^{-2} t \partial_x \partial_t u_1(t) \|_{L^2_x(\bbR)} \lesssim \varepsilon^3 \jap{t}^{-\frac{1}{2}}$. All other terms in the decomposition of $(\partial_t Z u_1)(t)$ can be dealt with similarly, thus establishing~\eqref{equ:intro_weighted_Zu1}.

\medskip 

Lastly, we note that the proof of Corollary~\ref{cor:asymptotics} follows the approach in~\cite{LS05_1, LS05_2} and is based on a version of the ODE argument in the proof of Theorem~\ref{thm:main}, which takes into account the oscillations of the solution to reveal the precise asymptotics. 

\medskip

\noindent {\it Organization of the paper:} The remainder of this paper is structured as follows. In Section~\ref{sec:preliminaries} we recall some Littlewood-Paley theory and collect the pointwise-in-time local decay estimates for the Klein-Gordon propagator that play a crucial role in our argument. Then we provide the proofs of the exterior decay estimate~\eqref{equ:main_exterior_estimate} in Section~\ref{sec:exterior_region} and of the interior decay estimate~\eqref{equ:main_interior_estimate} in Section~\ref{sec:interior_region}. Finally, we establish the asymptotic behavior of the solution in Section~\ref{sec:asymptotics}.
 
\medskip

\noindent {\it Acknowledgements:} The authors are grateful to Timothy Candy for helpful discussions.

\section{Preliminaries} \label{sec:preliminaries}

\subsection{Notation}

 We use the standard convention to denote by $C > 0$ an absolute constant whose value may change from line to line. We write $X \lesssim Y$ if $X \leq C Y$. In order to avoid writing out absolute constants in some identities, we use the notation $X \simeq Y$ to indicate that $|X| \lesssim |Y| \lesssim |X|$. Moreover, we write $X \ll Y$ if the implicit constant is to be regarded suitably small. We use $\jap{x} := (1 + x^2)^{\frac{1}{2}}$ with analogous definitions for $\jap{t}$ and $\jap{\xi}$.

 The Fourier transform and its inverse are defined by
 \begin{equation*}
  \hat{u}(\xi) = (\calF u)(\xi) := \frac{1}{\sqrt{2\pi}} \int_{\bbR} e^{-i x \xi} u(x) \, \ud x, \quad \text{and} \quad (\calF^{-1} u)(x) := \frac{1}{\sqrt{2\pi}} \int_{\bbR} e^{i x \xi} u(\xi) \, \ud \xi.
 \end{equation*}
 We will employ the spatial Fourier transform both in the rectilinear coordinates $(t,x)$ and in the hyperbolic coordinates $(\rho, y)$. Unless it is clear from the context, we will use subscripts to indicate the spatial variable with respect to which the Fourier transform is taken. 
 
 Finally, we set $\jap{\nabla} := \calF_{x}^{-1} \jap{\xi} \calF_{x}$ and denote by $e^{\pm i t \jap{\nabla}} := \calF_x^{-1} e^{\pm i t \jap{\xi}} \calF_x$ the propagators for the linear Klein-Gordon equation.

\subsection{Littlewood-Paley theory}

 In the proof of the interior decay estimate and the asymptotic behavior of the solution in Section 4 and 5, we will use some basic Littlewood-Paley theory while working in hyperbolic coordinates $(\rho, y)$. We denote by $\eta$ the spatial Fourier variable with respect to $y$. Let $\varphi \in C_c^\infty(\bbR)$ be a smooth bump function satisfying $\varphi(\eta) = 1$ for $|\eta| \leq 1$ and $\varphi(\eta) = 0$ for $|\eta| \geq 2$. Then $\psi(\eta) := \varphi(\eta) - \varphi(2 \eta)$ is a smooth bump function whose support is localized to $|\eta| \sim 1$. 
 For any integer $k \in \bbZ$ we define the usual dyadic Littlewood-Paley projections
 \begin{equation} \label{equ:definition_Pk}
  P_k w := \calF_y^{-1} \Bigl( \psi( {\textstyle \frac{\eta}{2^k} } ) (\calF_y w)(\eta) \Bigr).
 \end{equation}
 Moreover, for any $\mu > \lambda > 0$ we introduce the projections
 \begin{align*}
  P_{\leq \lambda} w &:= \calF_y^{-1} \Bigl( \varphi( {\textstyle \frac{\eta}{\lambda} } ) (\calF_y w)(\eta) \Bigr)
 \end{align*}
 with corresponding definitions for $P_{>\lambda}$ and $P_{[\lambda, \mu]}$. 
 
 The following standard Bernstein estimates will be used freely in the sequel 
 \begin{align*}
  \| P_k w \|_{L^p_y(\bbR)} \lesssim 2^{(\frac{1}{q} - \frac{1}{p})k} \| P_k w \|_{L^q_y(\bbR)}, \qquad \| \partial_y P_k w \|_{L^p_y(\bbR)} \simeq 2^k \|P_k w\|_{L^p_y(\bbR)}, \quad 1 \leq q \leq p \leq \infty.
 \end{align*}
 Moreover, we will need $L^\infty_y$ bounds for frequency localizations of the coefficients of the cubic nonlinearities on the right-hand side of~\eqref{equ:nlkg_hyperbolic_coord}.
 \begin{lemma}
  For any $k \in \bbZ$, $N \in \bbN$, and $\rho \geq 1$ we have that
  \begin{align} 
   \Bigl\| P_k \Bigl( \frac{\beta_0}{\cosh(\cdot)} \Bigr) \Bigr\|_{L^\infty_y(\bbR)} &\lesssim_N 2^{-N k},    \label{equ:Pk_beta_const_Linfty} \\
   \Bigl\| P_k \Bigl( \frac{\beta( \rho \sinh(\cdot) )}{\cosh(\cdot)} \Bigr) \Bigr\|_{L^\infty_y(\bbR)} &\lesssim 2^k \rho^{-1}.    \label{equ:Pk_beta_variable_Linfty}
  \end{align}
 \end{lemma}
 \begin{proof}
  The estimate~\eqref{equ:Pk_beta_const_Linfty} is a consequence of standard Bernstein estimates
  \begin{align*}
   \Bigl\| P_k \Bigl( \frac{\beta_0}{\cosh(\cdot)} \Bigr) \Bigr\|_{L^\infty_y} \lesssim 2^{-Nk} \bigl\| \partial_y^N \cosh^{-1}(\cdot) \bigr\|_{L^\infty_y} \lesssim_N 2^{-Nk}.
  \end{align*}
  For the proof of~\eqref{equ:Pk_beta_variable_Linfty} we start from the definition~\eqref{equ:definition_Pk} of the projection $P_k$ and compute that
  \begin{align*}
   \biggl\| P_k \biggl( \frac{\beta(\rho \sinh(\cdot))}{\cosh(\cdot)} \biggr) \biggr\|_{L^\infty_y} 
   \lesssim \biggl( \int_{\bbR_\eta} \, \Bigl| \psi \Bigl( \frac{\eta}{2^k} \Bigr) \Bigr| \, \ud \eta \biggr) \biggl\| \calF_y \biggl( \frac{\beta(\rho \sinh(\cdot))}{\cosh(\cdot)} \biggr)  \biggr\|_{L^\infty_\eta} \lesssim 2^k \biggl\| \frac{\beta(\rho \sinh(\cdot))}{\cosh(\cdot)} \biggr\|_{L^1_y}.
  \end{align*}
  Then we exploit that $\beta(x)$ is rapidly decaying by assumption to find by a change of variables that
  \begin{align*}
   \biggl\| \frac{\beta(\rho \sinh(\cdot))}{\cosh(\cdot)} \biggr\|_{L^1_y} &\lesssim \int_{\bbR} | \beta(\rho \sinh(y)) | \, \ud y \lesssim \int_{\bbR} \frac{C(\beta)}{(1 + \rho |\sinh(y)|)^2} \, \ud y \lesssim \int_{\bbR} \frac{C(\beta)}{(1 + \rho |y|)^2} \, \ud y \lesssim \rho^{-1},
  \end{align*}
  which together with the previous bound yields~\eqref{equ:Pk_beta_variable_Linfty}.
 \end{proof}

\subsection{Pointwise-in-time local decay estimates for the free Klein-Gordon propagator} \label{subsec:weighted_L2_decay}

In this subsection we collect the pointwise-in-time local decay estimates for the free Klein-Gordon propagator that play an essential role in Section~\ref{sec:interior_region} in the derivation of slowly growing energy estimates of a Lorentz boost of the solution to~\eqref{equ:nlkg}. For the convenience of the reader we provide simple Fourier analysis proofs here, but we emphasize that such estimates can be established for much more general classes of Hamiltonian operators via resolvent or positive commutator methods, see \cite{Rauch78, KJ79, Jensen80, Jensen84, JMP84, GLS16, Schl07} and references therein.

\begin{lemma} \label{lem:weighted_decay}
 For any $a \geq 1$ and $b \geq 0$, there exists an absolute constant $C \geq 1$ such that
 \begin{equation} \label{equ:weighted_decay}
  \sup_{t \in \bbR} \, \jap{t}^{\frac{1}{2}} \bigl\| \jap{x}^{-a} \jap{\nabla}^{-b} e^{\pm i t \jap{\nabla}} \jap{x}^{-a} \bigr\|_{L^2_x(\bR) \to L^2_x(\bR)} \leq C.
 \end{equation}
\end{lemma}
\begin{proof}
 We only consider the case $a = 1$ and $b=0$ for $e^{+it\langle \nabla \rangle}$. The other cases can be treated analogously. By time-reversal symmetry we may assume that $t$ is non-negative and we may additionally assume that $t \geq 1$, since otherwise the estimate~\eqref{equ:weighted_decay} is trivial.

 Let $\varphi \in C_c^\infty(\bR)$ be a smooth, non-negative bump function satisfying $\varphi(\xi) = 1$ for $|\xi| \leq 1$ and $\varphi(\xi) = 0$ for $|\xi| \geq 2$. Now let $f \in L^2_x(\bR)$. By density of $\calS(\bR)$ in $L^2_x(\bR)$ we may assume that $f$ is in fact smooth and rapidly decaying so that all identities and integration by parts in what follows are justified. We have
 \begin{align*}
  \jap{x}^{-1} e^{it\jap{\nabla}} \jap{x}^{-1} f &\simeq \jap{x}^{-1} \int_{\bbR} e^{ix\xi} e^{it\jap{\xi}} \widehat{ (\jap{x}^{-1} f) }(\xi) \, \ud \xi \\
  &= \jap{x}^{-1} \int_{\bbR} e^{ix\xi} e^{it\jap{\xi}} \varphi( t^{\frac{1}{2}} \xi ) \widehat{ (\jap{x}^{-1} f) }(\xi) \, \ud \xi \\
  &\quad + \jap{x}^{-1} \int_{\bbR} e^{ix\xi} e^{it\jap{\xi}} \bigl(1 - \varphi(t^{\frac{1}{2}} \xi) \bigr) \widehat{ (\jap{x}^{-1} f) }(\xi) \, \ud \xi \\
  &\equiv I + II.
 \end{align*}
 In order to estimate the term $I$ we use one-dimensional Sobolev embedding $H^1_\xi(\bbR) \hookrightarrow L^\infty_\xi(\bbR)$ in the Fourier variable to obtain that
 \begin{align*}
  |I| &\lesssim \jap{x}^{-1} \biggl( \int_{\bbR} |\varphi(t^{\frac{1}{2}} \xi)| \, \ud \xi \biggr) \bigl\| \widehat{ (\jap{x}^{-1} f) }(\xi) \bigr\|_{L^\infty_\xi(\bbR)} \lesssim \jap{x}^{-1} t^{-\frac{1}{2}} \bigl\| \widehat{ (\jap{x}^{-1} f) }(\xi) \bigr\|_{H^1_\xi(\bbR)}.
 \end{align*}
 By Plancherel's theorem we have
 \begin{align*}
  \bigl\| \widehat{ (\jap{x}^{-1} f) }(\xi) \bigr\|_{H^1_\xi(\bbR)} &\lesssim \bigl\| \widehat{ (\jap{x}^{-1} f) }(\xi) \bigr\|_{L^2_\xi(\bbR)} + \bigl\| \partial_\xi \widehat{ (\jap{x}^{-1} f) }(\xi) \bigr\|_{L^2_\xi(\bbR)} \\
  &\simeq \bigl\| \widehat{ (\jap{x}^{-1} f) }(\xi) \bigr\|_{L^2_\xi(\bbR)} + \bigl\| {\textstyle \widehat{ (\frac{-i x}{\jap{x}} f)} }(\xi) \bigr\|_{L^2_\xi(\bbR)} \\
  &\lesssim \bigl\| \jap{x}^{-1} f \bigr\|_{L^2_x(\bbR)} + \bigl\| {\textstyle \frac{-ix}{\jap{x}}} f \bigr\|_{L^2_x(\bbR)} \\
  &\lesssim \|f\|_{L^2_x(\bR)}.
 \end{align*}
 Since $\|\jap{x}^{-1}\|_{L^2_x(\bbR)} \lesssim 1$, it follows that
 \begin{equation*}
  \| I \|_{L^2_x(\bbR)} \lesssim t^{-\frac{1}{2}} \|f\|_{L^2_x(\bbR)}.
 \end{equation*}
 In order to bound the term $II$ we use the identity (for $\xi \neq 0$)
 \[
  e^{it\jap{\xi}} = \frac{1}{i t} \frac{\jap{\xi}}{\xi} \frac{\partial}{\partial \xi} \bigl( e^{i t \jap{\xi}} \bigr)
 \]
 and integrate by parts (in $\xi$) to find that
 \begin{align*}
  II &= \jap{x}^{-1} \int_{\bbR} e^{ix\xi} e^{it\jap{\xi}} \bigl(1 - \varphi(t^{\frac{1}{2}} \xi) \bigr) \widehat{ (\jap{x}^{-1} f) }(\xi) \, \ud \xi \\
  &= - \frac{1}{it} \jap{x}^{-1} \int_{\bbR} e^{it\jap{\xi}} \frac{\partial}{\partial \xi} \biggl( e^{ix\xi} \frac{\jap{\xi}}{\xi} \bigl(1 - \varphi(t^{\frac{1}{2}} \xi) \bigr) \widehat{ (\jap{x}^{-1} f) }(\xi) \biggr) \, \ud \xi \\
  &= - \frac{1}{t} \frac{x}{\jap{x}} \int_{\bbR} e^{it\jap{\xi}} e^{ix\xi} \frac{\jap{\xi}}{\xi} \bigl(1 - \varphi(t^{\frac{1}{2}} \xi) \bigr) \widehat{ (\jap{x}^{-1} f) }(\xi) \, \ud \xi \\
  &\quad + \frac{1}{i t} \jap{x}^{-1} \int_{\bbR} e^{it\jap{\xi}} e^{ix\xi} \frac{1}{\jap{\xi} \xi^2} \bigl(1 - \varphi(t^{\frac{1}{2}} \xi) \bigr) \widehat{ (\jap{x}^{-1} f) }(\xi) \, \ud \xi \\
  &\quad + \frac{t^{\frac{1}{2}}}{i t} \jap{x}^{-1} \int_{\bbR} e^{it\jap{\xi}} e^{ix\xi} \frac{\jap{\xi}}{\xi} \varphi'(t^{\frac{1}{2}} \xi) \widehat{ (\jap{x}^{-1} f) }(\xi) \, \ud \xi \\
  &\quad - \frac{1}{it} \jap{x}^{-1} \int_{\bbR} e^{it\jap{\xi}} e^{ix\xi} \frac{\jap{\xi}}{\xi} \bigl(1 - \varphi(t^{\frac{1}{2}} \xi) \bigr) \widehat{ {\textstyle ( \frac{-i x}{\jap{x}} f) } }(\xi) \, \ud \xi \\
  &\equiv II_A + II_B + II_C + II_D.
 \end{align*}

 In order to estimate the term $II_A$ we first note that in view of the support properties of $\varphi$ and since $t \geq 1$, we have uniformly for all $\xi \in \bbR$ that
 \begin{equation*}
  \Bigl| \frac{\jap{\xi}}{\xi} (1 - \varphi(t^{\frac{1}{2}} \xi)) \Bigr| \lesssim t^{\frac{1}{2}}.
 \end{equation*}
 Thus, by Plancherel it holds that
 \begin{align*}
  \biggl\| \int_{\bbR} e^{i x \xi} e^{i t \jap{\xi}} \frac{\jap{\xi}}{\xi} \bigl(1 - \varphi(t^{\frac{1}{2}} \xi) \bigr) \widehat{ (\jap{x}^{-1} f) }(\xi) \, \ud \xi \biggr\|_{L^2_x(\bbR)} &= \biggl\| e^{i t \jap{\xi}} \frac{\jap{\xi}}{\xi} \bigl(1 - \varphi(t^{\frac{1}{2}} \xi) \bigr) \widehat{ (\jap{x}^{-1} f) }(\xi) \biggr\|_{L^2_{\xi}(\bbR)} \\
  &\lesssim t^{\frac{1}{2}} \bigl\| \widehat{ (\jap{x}^{-1} f) }(\xi) \bigr\|_{L^2_\xi(\bbR)} \\
  &\lesssim t^{\frac{1}{2}} \|f\|_{L^2_x(\bbR)}.
 \end{align*}
 We may therefore bound the term $II_A$ by
 \begin{align*}
  \| II_A \|_{L^2_x(\bbR)} &\lesssim \frac{1}{t} \biggl\| e^{i t \jap{\xi}} \frac{\jap{\xi}}{\xi} \bigl(1 - \varphi(t^{\frac{1}{2}} \xi) \bigr) \widehat{ (\jap{x}^{-1} f) }(\xi) \biggr\|_{L^2_{\xi}(\bbR)} \lesssim \frac{1}{t^{\frac{1}{2}}} \|f\|_{L^2_x(\bbR)}.
 \end{align*}
 For the term $II_B$ we first observe that in view of the support properties of $\varphi$ and using again one-dimensional Sobolev embedding $H^1_{\xi}(\bbR) \hookrightarrow L^\infty_{\xi}(\bbR)$, we have
 \begin{align*}
  \biggl\| \int_{\bbR} e^{it\jap{\xi}} e^{ix\xi} \frac{1}{\jap{\xi} \xi^2} \bigl(1 - \varphi(t^{\frac{1}{2}} \xi) \bigr) \widehat{ (\jap{x}^{-1} f) }(\xi) \, \ud \xi \biggr\|_{L^\infty_x(\bbR)} &\lesssim \biggl( \int_{ \{ |\xi| \geq t^{-\frac{1}{2}} \} } \frac{1}{\xi^2} \, \ud \xi \biggr) \bigl\| \widehat{ (\jap{x}^{-1} f) }(\xi) \bigr\|_{L^\infty_\xi(\bbR)} \\
  &\lesssim t^{\frac{1}{2}} \|f\|_{L^2_x(\bbR)}.
 \end{align*}
 Hence, we obtain that
 \begin{align*}
  \| II_B \|_{L^2_x(\bbR)} \lesssim \frac{1}{t} \bigl\| \jap{x}^{-1} \bigr\|_{L^2_x(\bbR)} \biggl\| \int_{\bbR} e^{it\jap{\xi}} e^{ix\xi} \frac{1}{\jap{\xi} \xi^2} \bigl(1 - \varphi(t^{\frac{1}{2}} \xi) \bigr) \widehat{ (\jap{x}^{-1} f) }(\xi) \, \ud \xi \biggr\|_{L^\infty_x(\bbR)} \lesssim \frac{1}{t^{\frac{1}{2}}} \|f\|_{L^2_x(\bbR)}.
 \end{align*}
 Next, to estimate the term $II_C$ we use that $t \geq 1$ and the support properties of $\varphi$ to find that
 \begin{align*}
  \| II_C \|_{L^2_x(\bbR)} &\lesssim \frac{1}{t^{\frac{1}{2}}} \bigl\| \jap{x}^{-1} \bigr\|_{L^2_x(\bbR)} \biggl( \int_{\bbR} \frac{\jap{\xi}}{|\xi|} |\varphi'(t^{\frac{1}{2}} \xi)| \, \ud \xi \biggr) \bigl\| \widehat{ (\jap{x}^{-1} f) }(\xi) \bigr\|_{L^\infty_\xi(\bbR)} \\
  &\lesssim \biggl( \int_{\bbR} \frac{1}{|t^{\frac{1}{2}} \xi|} |\varphi'(t^{\frac{1}{2}} \xi)| \, \ud \xi \biggr) \|f\|_{L^2_x(\bbR)} \\
  &\lesssim \frac{1}{t^{\frac{1}{2}}} \|f\|_{L^2_x(\bbR)},
 \end{align*}
 where in the last line we made a change of variables. Finally, in order to bound the term $II_D$ we use Plancherel again and proceed analogously to the treatment of the term $II_A$.
\end{proof}

In the presence of a spatial derivative, we obtain an improved pointwise-in-time local decay estimate.

\begin{lemma} \label{lem:weighted_derivative_decay}
 There exists an absolute constant $C \geq 1$ such that
 \begin{equation} \label{equ:weighted_derivative_decay}
  \sup_{t \in \bbR} \, \jap{t}^{\frac{3}{2}} \Bigl\| \jap{x}^{-2} \frac{\partial_x}{\jap{\nabla}} e^{\pm i t \jap{\nabla}} \jap{x}^{-2} \Bigr\|_{L^2_x(\bR) \to L^2_x(\bR)} \leq C.
 \end{equation}
\end{lemma}
\begin{proof}
 We may again assume that $t \geq 1$. Let $f \in L^2_x(\bbR)$. Using the identity
 \[
  \frac{i\xi}{\jap{\xi}} e^{it\jap{\xi}} = \frac{1}{t} \frac{\partial}{\partial \xi} \bigl( e^{it\jap{\xi}} \bigr),
 \]
 we integrate by parts (in $\xi$) once to write
 \begin{align*}
  \jap{x}^{-2} \frac{\partial_x}{\jap{\nabla}} e^{\pm i t \jap{\nabla}} \jap{x}^{-2} f &\simeq \jap{x}^{-2} \int_{\bbR} e^{ix\xi} \frac{i\xi}{\jap{\xi}} e^{it\jap{\xi}} \widehat{ (\jap{x}^{-2} f) }(\xi) \, \ud \xi \\
  &= \frac{1}{t} \jap{x}^{-2} \int_{\bbR} e^{i x \xi} \frac{\partial}{\partial \xi} \bigl( e^{it\jap{\xi}} \bigr) \widehat{ (\jap{x}^{-2} f) }(\xi) \, \ud \xi \\
  &= -\frac{1}{t} \frac{ix}{\jap{x}^2} \int_{\bbR} e^{ix\xi} e^{i t \jap{\xi}} \widehat{ (\jap{x}^{-2} f) }(\xi) \, \ud \xi \\
  &\quad - \frac{1}{t} \jap{x}^{-2} \int_{\bbR} e^{ix\xi} e^{i t \jap{\xi}} \widehat{ {\textstyle ( \frac{-ix}{\jap{x}^2} f) } }(\xi) \, \ud \xi \\
  &\equiv I + II.
 \end{align*}
 Both for term $I$ and for term $II$ we may now use the stationary phase type argument as in the proof of the decay estimate~\eqref{equ:weighted_decay} to infer an additional decay factor of $t^{-\frac{1}{2}}$, which yields the asserted overall $t^{-\frac{3}{2}}$ decay. This finishes the proof.
\end{proof}

\begin{lemma} \label{lem:weighted_derivative_decay2}
 There exists an absolute constant $C \geq 1$ such that
 \begin{equation} \label{equ:weighted_derivative_decay2}
  \sup_{t \in \bbR} \, \jap{t}^{\frac{3}{2}} \bigl\| \jap{x}^{-2} \partial_x e^{\pm i t \jap{\nabla}} \jap{x}^{-2} \bigr\|_{H^1_x(\bR) \to L^2_x(\bR)} \leq C.
 \end{equation}
\end{lemma}
The proof of Lemma~\ref{lem:weighted_derivative_decay2} is analogous to the proof of Lemma~\ref{lem:weighted_derivative_decay}. We omit the details.

\section{Exterior Region} \label{sec:exterior_region}

In this section we use an argument of Klainerman~\cite{Kl93} to establish the exterior decay estimate~\eqref{equ:main_exterior_estimate} as well as energy bounds that will be needed later for the treatment of the interior region. We follow the presentation in Candy-Lindblad~\cite{CL18}.

\medskip

We first recall that the Klein-Gordon energy momentum tensor 
\begin{align*}
  Q_{\alpha \beta}[\phi] = 2 (\partial_\alpha \phi) (\partial_\beta \phi) - m_{\alpha \beta} \bigl( (\partial_\mu \phi) (\partial^\mu \phi) - \phi^2 \bigr)
\end{align*}
satisfies
\begin{align*}
 \partial^\alpha Q_{\alpha \beta}[\phi] = 2 \bigl( (\Box + 1) \phi \bigr) (\partial_\beta \phi).
\end{align*}
Here $m = \text{diag} (+1, -1)$ denotes the standard Lorentzian metric on $\bR^{1+1}$ and we are using the standard conventions for raising/lowering and summing indices.

For any $T \geq 1$ we consider the domain
\begin{equation*}
 \calD_T = \bigl\{ (t,x) \in \bbR^{1+1} \, : \, \jap{x} \geq t, 1 \leq t \leq T \bigr\}
\end{equation*}
with boundary
\begin{equation*}
 S_T = \bigl\{ (T,x) \in \bbR^{1+1} \, : \, \jap{x} \geq T \bigr\} \cup \bigl\{ (\jap{x}, x) \, \colon \, \jap{x} \leq T \bigr\}.
\end{equation*}
Then for $T \geq 1$ and a non-negative weight $\omega \geq 0$, we define the energy
\begin{equation*}
 E_{ext, T}(\phi, \omega) = \int_{S_T} n^\alpha[S_T] Q_{\alpha 0}[\phi] \omega \, \ud x,
\end{equation*}
where
\begin{equation*}
 n^\alpha[S_T] \partial_\alpha = \left\{ \begin{array}{ll}
                                     \partial_t \quad &\text{ on } \, \bigl\{ \jap{x} \geq T, t = T \bigr\}, \\
                                     \partial_t + \frac{x}{\jap{x}} \partial_x \quad &\text{ on } \, \bigl\{ \jap{x} = t, 1 \leq t \leq T \bigr\}.
                                    \end{array}
                            \right.
\end{equation*}
We have
\begin{equation*}
 n^\alpha[S_T] Q_{\alpha 0}[\phi] = \left\{ \begin{array}{ll}
                                        (\partial_t \phi)^2 + (\partial_x \phi)^2 + \phi^2 \quad &\text{ on } \, \jap{x} > T \text{ and } t = T, \\
                                        (\partial_t \phi)^2 + (\partial_x \phi)^2 + 2 (\partial_t \phi)(\partial_x \phi) {\textstyle \frac{x}{t}} + \phi^2 &\text{ on } \, \jap{x} = t \text{ and } t < T,
                                       \end{array}
			       \right.
\end{equation*}
and note the coercivity property
\begin{align*}
 (\partial_t \phi)^2 + (\partial_x \phi)^2 + 2 (\partial_t \phi)(\partial_x \phi) {\textstyle \frac{x}{t}} + \phi^2 &= (\px \phi + {\textstyle \frac{x}{t}} \pt \phi)^2 + (\pt \phi)^2 {\textstyle \frac{\rho^2}{t^2} } + \phi^2 \\
 &= (\pt \phi + {\textstyle \frac{x}{t}} \px \phi)^2 + (\px \phi)^2 {\textstyle \frac{\rho^2}{t^2} } + \phi^2.
\end{align*}
Then the following weighted energy estimate holds, see \cite[Theorem 3]{Kl93} and \cite[Lemma 2.1]{CL18}.
\begin{lemma} \label{lem:exterior_energy_estimate}
 Let $1 \leq T < \infty$, $N \in \bbN$, and for $0 \leq j \leq N$ define the weights
 \begin{equation*}
  \omega_j = (t+|x|)^{N-j} (|x|-t+1)^j.
 \end{equation*}
 Then we have
 \begin{equation*}
  \sum_{|I| \leq N} E_{ext, T}(\partial^I \phi, \omega_{|I|})^{\frac{1}{2}} \lesssim \sum_{|I| \leq N} E_{ext, 1}(\partial^I \phi, \omega_{|I|})^{\frac{1}{2}} + \sum_{|I| \leq N} \int_1^T \biggl( \int_{\jap{x} \geq t} \bigl( (\Box + 1) \partial^I \phi \bigr)^2 \omega_{|I|} \, \ud x \biggr)^{\frac{1}{2}} \, \ud t.
 \end{equation*}
\end{lemma}

We now turn to deriving the exterior decay estimate~\eqref{equ:main_exterior_estimate} in Theorem~\ref{thm:main}. Let $u(t)$ be the solution to~\eqref{equ:nlkg} and recall the decomposition $u(t) = u_0(t) + u_1(t)$ from~\eqref{equ:decomposition_solution}. For any $T \geq 1$ and any $N \geq 2$, we define the exterior energy functional
\begin{equation*}
 \calE_{ext}(T) := \sum_{|I| \leq N}  E_{ext, T}(\partial^I u, \omega_{|I|}) + \sum_{0 \leq \ell \leq 1} \sum_{|I| \leq N} E_{ext, T}(\partial^I Z u_\ell, \omega_{|I|}).
\end{equation*}

Next we observe that if a function satisfies $E_{ext, T}(\phi, \omega_j) < \infty$ for $0 \leq j \leq N$, then an application of one-dimensional Sobolev embedding together with the fact that $\px \omega_j \lesssim \omega_j$ gives for any $\jap{x} \geq T$ and any $0 \leq j \leq N$ that
\begin{equation} \label{equ:pointwise_by_ext_energy}
 \bigl( \phi^2 \omega_j \bigr)(T, x) \lesssim E_{ext, T}(\phi, \omega_j).
\end{equation}
Using this pointwise estimate we compute that for any $T \geq 1$ and $\jap{x} \geq T$ there holds
\begin{align*}
 &\sum_{|I| \leq N} \bigl( \partial^I ( \beta_0 u^3 + \beta(x) u^3 ) \bigr)^2 \omega_{|I|} \\
 &\quad \lesssim_{\beta_0, \beta(x)} \sum_{|I| \leq N} w_{|I|} \sum_{|J_1| + |J_2| + |J_3| \leq |I|} (\partial^{J_1} u)^2 (\partial^{J_2} u)^2 (\partial^{J_3} u)^2 \\
 &\quad \lesssim \sum_{|I| \leq N}\sum_{|J_1| + |J_2| + |J_3| \leq |I|} \frac{\omega_{|I|}}{\omega_{|J_1|} \omega_{|J_2|} \omega_{|J_3|}}  E_{ext,T}(\partial^{J_1} u, \omega_{|J_1|}) E_{ext,T}(\partial^{J_2} u, \omega_{|J_2|}) (\partial^{J_3} u)^2 \omega_{|J_3|} \\
 &\quad \lesssim t^{-2N} \Bigl( \sup_{|I| \leq N} E_{ext, T}(\partial^I u, \omega_{|I|}) \Bigr)^2 \sum_{|I| \leq N} (\partial^I u)^2 \omega_{|I|}.
\end{align*}
Moreover, exploiting that in the exterior region $t \leq \jap{x}$, we find in a similar manner that
\begin{align*}
 &\sum_{|I| \leq N} \bigl( \partial^I Z ( \beta_0 u^3 ) \bigr)^2 \omega_{|I|} + \sum_{|I| \leq N} \bigl( \partial^I Z ( \beta(x) u^3 ) \bigr)^2 \omega_{|I|} \\
 &\quad \lesssim \Bigl( |\beta_0| + \bigl\| \jap{x} \jap{\nabla}^{N+1} \beta(x) \bigr\|_{L^\infty_x} \Bigr) \sum_{0 \leq k_1, k_2, k_3 \leq 1} \sum_{|I| \leq N} \omega_{|I|} \sum_{J_1 + J_2 + J_3 \leq I} (\partial^{J_1} Z^{k_1} u)^2 (\partial^{J_2} Z^{k_2} u)^2  (\partial^{J_3} Z^{k_3} u)^2 \\
 &\quad \lesssim t^{-2N} \biggl( \sup_{|I| \leq N} \, E_{ext, T}(\partial^I u, \omega_{|I|}) + \sup_{ \substack{ |I| \leq N \\ 0 \leq \ell \leq 2 } } \, E_{ext, T}(\partial^I Z u_{\ell}, \omega_{|I|}) \biggr)^2 \sum_{|I| \leq N} \Bigl( (\partial^I u)^2 + \sum_{0 \leq \ell \leq 2} (\partial^I Z u_\ell)^2 \Bigr) \omega_{|I|}.
\end{align*}
Combining the above estimates with the exterior energy estimate from Lemma~\ref{lem:exterior_energy_estimate} and the fact that $Z$~commutes with $\Box + 1$, we obtain for any $T \geq 1$~that
\begin{equation} \label{equ:goal_exterior_energy_estimate}
 \calE_{ext}(T)^{\frac{1}{2}} \lesssim \calE_{ext}(1)^{\frac{1}{2}} + \int_1^T t^{-N} \calE_{ext}(t)^{\frac{3}{2}} \, \ud t.
\end{equation}
By a standard continuity argument we infer from~\eqref{equ:goal_exterior_energy_estimate} the following.
\begin{proposition}[Exterior Energy Estimate] \label{prop:exterior_energy_estimate}
 Let $N \geq 2$. There exists a constant $\varepsilon_1 > 0$ such that if the initial data satisfy
 \begin{equation*}
  \varepsilon := \bigl\| \jap{x}^{1+\frac{N}{2}} f \bigr\|_{H^{N+2}_x} + \bigl\| \jap{x}^{1+\frac{N}{2}} g \bigr\|_{H^{N+1}_x} \leq \varepsilon_1,
 \end{equation*}
 then for every $T \geq 1$ we have
 \begin{equation} \label{equ:exterior_energy_estimate}
  \calE_{ext}(T) \lesssim \calE_{ext}(1) \lesssim \varepsilon^2.
 \end{equation}
 In particular, we conclude that in the exterior region $1 \leq t \leq \jap{x}$ it holds that
 \begin{equation} \label{equ:exterior_decay_estimate}
  |u(t,x)| \lesssim \jap{x}^{-\frac{N}{2}} \varepsilon.
 \end{equation}
\end{proposition}

\section{Interior Region} \label{sec:interior_region}

In this section we derive the interior decay estimate~\eqref{equ:main_interior_estimate} for the solution $u(t)$ to~\eqref{equ:nlkg}.

\begin{proposition}[Interior Decay Estimate] \label{prop:interior_decay}
 Let $N \geq 2$. There exist constants $0 < \varepsilon_0 \ll 1$ and $C_0 \gg 1$ such that if the initial data $(f,g)$ satisfy
 \begin{equation} \label{equ:data_assumption_interior_decay_prop}
  \varepsilon := \bigl\| \jap{x}^{1+\frac{N}{2}} f \bigr\|_{H^{N+2}_x(\bbR)} + \bigl\| \jap{x}^{1+\frac{N}{2}} g \bigr\|_{H^{N+1}_x(\bbR)} \leq \varepsilon_0,
 \end{equation}
 then we have for the solution $u(t)$ to~\eqref{equ:nlkg} that
 \begin{equation} \label{equ:interior_decay_estimate}
  \sup_{t \geq 1} \, \sup_{\jap{x} \leq t} \, t^{\frac{1}{2}} |u(t,x)| \leq C_0 \varepsilon.
 \end{equation}
\end{proposition}

The proof of Proposition~\ref{prop:interior_decay} proceeds via the following bootstrap argument: There exist constants $0 < \varepsilon_0 \ll 1$ small enough and $C_0 \gg 1$ sufficiently large such that, given any $R \geq 1$ and any initial data satisfying~\eqref{equ:data_assumption_interior_decay_prop}, we may conclude that
\begin{align}
 \sup_{1 \leq \rho \leq R} \,\sup_{(t,x) \in H_\rho} \, t^{\frac{1}{2}} |u(t,x)| &\leq 2 C_0 \varepsilon \label{equ:interior_bootstrap_assumption} \\
 &\qquad \qquad \Longrightarrow \quad \sup_{1 \leq \rho \leq R} \sup_{(t,x) \in H_\rho} \, t^{\frac{1}{2}} |u(t,x)| \leq C_0 \varepsilon.   \label{equ:interior_bootstrap_conclusion}
\end{align}

\medskip

We will make use of several constants whose relations we now specify. We denote by $0 < \delta \ll 1$ a small absolute constant. Then we will choose $0 < \varepsilon_0 \ll 1$ sufficiently small and $C_0 \gg 1$ sufficiently large such that, in particular, it holds that
\begin{equation*}
 C_0 \varepsilon_0 \ll \delta \ll 1.
\end{equation*}
Moreover, we suppose that $\varepsilon_0 \leq \varepsilon_1$, which will allow us to invoke the estimates from Proposition~\ref{prop:exterior_energy_estimate} for the exterior region.

\medskip

In the course of the argument we consider for any $R \geq 1$ the domain
\begin{equation*}
 \calD^R = \bigl\{ (t,x) \in \bbR^{1+1} \, : \, t^2 - x^2 \leq R^2, t \geq 1 \bigr\}.
\end{equation*}
Furthermore, for any $R \geq 1$ and any $t \geq 1$ we define the fixed time-slice
\begin{equation*}
 \calS_t^R = \calD^R \cap \{ t \} \times \bbR
\end{equation*}
and denote by $\chi_{S_t^R}(x)$ a sharp cut-off function given by
\begin{equation*}
 \chi_{\calS_t^R}(x) = \left\{ \begin{array}{ll}
                                          1, \quad (t,x) \in \calS_t^R, \\
                                          0, \quad \text{otherwise}.
                                         \end{array} \right.
\end{equation*}

\subsection{Interior Energy Estimate} \label{subsec:interior_energy_estimate}

The goal of this subsection is to derive a slow growth estimate for the energy of one Lorentz boost of the solution to~\eqref{equ:nlkg} in the interior region, which we foliate in terms of the hyperboloids $H_\rho$, $\rho \geq 1$. The associated energies are then given by
\begin{equation*}
 E_{int, \rho}(\phi) := \int_{H_\rho} n^\alpha[H_\rho] Q_{\alpha 0}[\phi] \, \ud x = \int_{H_\rho} \bigl( (\partial_t \phi)^2 + (\partial_x \phi)^2 + 2 (\partial_t \phi) (\partial_x \phi) {\textstyle \frac{x}{t}} + \phi^2 \bigr) \, \ud x,
\end{equation*}
where the hyperboloid $H_\rho$ is parametrized by $x$. By \cite[Section 7.6]{H97} we have
\begin{equation*}
 \frac{\ud}{\ud \rho} E_{int, \rho}(\phi) = 2 \int_{H_\rho} \bigl( (\Box + 1) \phi \bigr) (\partial_t \phi) \frac{\rho}{t} \, \ud x.
\end{equation*}
Note that a computation gives
\begin{equation*}
 (\partial_t \phi)^2 + (\partial_x \phi)^2 + 2 (\partial_t \phi) (\partial_x \phi) {\textstyle \frac{x}{t}} + \phi^2 = {\textstyle \frac{\rho^2}{2 t^2}} \bigl( (\pt \phi)^2 + (\px \phi)^2 \bigr) + {\textstyle \frac{\rho^2}{2 t^2}} \bigl( (\partial_\rho \phi)^2 + {\textstyle \frac{1}{\rho^2}} (\partial_y \phi)^2 \bigr) + \phi^2.
\end{equation*}
Passing to hyperbolic coordinates it therefore holds that
\begin{equation*}
 E_{int, \rho}(\phi) = \int_{\bbR_y} \Bigl( {\textstyle \frac{1}{2}} \cosh^{-2}(y) \bigl( (\pt \phi)^2 + (\px \phi)^2 + (\partial_\rho \phi)^2 + {\textstyle \frac{1}{\rho^2}} (\partial_y \phi)^2 \bigr) + \phi^2 \Bigr) \, \rho \cosh(y) \, \ud y.
\end{equation*}

Let $u(t)$ be the solution to~\eqref{equ:nlkg} and recall the decomposition $u(t) = u_0(t) + u_1(t)$ from~\eqref{equ:decomposition_solution}. For any $\rho \geq 1$ we introduce the interior energy functional
\begin{equation}
 \calE_{int}(\rho) := E_{int, \rho}(u) + \sum_{0 \leq \ell \leq 1} E_{int, \rho}(Z u_\ell).
\end{equation}
In particular, note that $\calE_{int}(\rho)$ bounds the energy $E_{int, \rho}(Z u)$ of one Lorentz boost of the solution. The next proposition provides a slow growth estimate for this interior energy functional.
\begin{proposition}[Interior Energy Estimate] \label{prop:interior_energy_estimate}
 Let $N \geq 2$ and $R \geq 1$. Assume that the initial data satisfy
 \begin{equation*}
  \varepsilon := \bigl\| \jap{x}^{1+\frac{N}{2}} f \bigr\|_{H^{N+2}_x(\bbR)} + \bigl\| \jap{x}^{1+\frac{N}{2}} g \bigr\|_{H^{N+1}_x(\bbR)} \leq \varepsilon_0.
 \end{equation*}
 Moreover, suppose that the solution $u(t)$ to~\eqref{equ:nlkg} satisfies
 \begin{equation} \label{equ:growth_interior_energies_Linfty_assumption}
  \sup_{t \geq 1} \, t^{\frac{1}{2}} \bigl\| \chi_{\calS_t^R} u(t) \bigr\|_{L^\infty_x(\bbR)} \leq 2 C_0 \varepsilon.
 \end{equation}
 Then we have
 \begin{equation} \label{equ:interior_E_int_goal}
  \sup_{1 \leq \rho \leq R} \, \rho^{-2\delta} \calE_{int}(\rho) \lesssim \varepsilon^2,
 \end{equation}
 where the implicit constant is independent of $C_0$.
\end{proposition}

A key ingredient for the proof of Proposition~\ref{prop:interior_energy_estimate} are weighted energy estimates for the component $u_1(t)$ of the solution $u(t)$ to~\eqref{equ:nlkg}. Their proof combines the pointwise-in-time local decay estimates for the Klein-Gordon propagator from Subsection~\ref{subsec:weighted_L2_decay} with the spatial localization provided by the coefficient $\beta(x)$ of the nonlinearity for $(\Box + 1) u_1$.

\begin{lemma}[Weighted Energy Estimates for $u_1(t)$]  \label{lem:weighted_bounds_u1}
 Let $R \geq 1$ and suppose that the assumptions of Proposition~\ref{prop:interior_energy_estimate} are in place.
 Then we have uniformly for all $t \geq 1$ that
 \begin{align}
  \bigl\| \jap{x}^{-2} \chi_{\calS_t^R} u_1(t) \bigr\|_{L^2_x(\bbR)} &\lesssim \frac{(2 C_0 \varepsilon)^3}{t^{\frac{1}{2}}}, \label{equ:weighted_bound_u1} \\
  \bigl\| \jap{x}^{-2} \chi_{\calS_t^R} \partial_x u_1(t) \bigr\|_{L^2_x(\bbR)} &\lesssim \frac{(2 C_0 \varepsilon)^3}{t^{\frac{3}{2}}}, \label{equ:weighted_bound_dx_u1} \\
  \bigl\| \jap{x}^{-2} \chi_{\calS_t^R} \partial_x^2 u_1(t) \bigr\|_{L^2_x(\bbR)} &\lesssim \frac{(2 C_0 \varepsilon)^3}{t^{\frac{3}{2} - C(2 C_0 \varepsilon)^2}}, \label{equ:weighted_bound_dx2_u1} \\
  \bigl\| \jap{x}^{-2} \chi_{\calS_t^R} \partial_x \partial_t u_1(t) \bigr\|_{L^2_x(\bbR)} &\lesssim \frac{(2 C_0 \varepsilon)^3}{t^{\frac{3}{2} - C(2 C_0 \varepsilon)^2}}, \label{equ:weighted_bound_dx_dt_u1}
 \end{align}
 for some absolute constant $C > 0$.
\end{lemma}
\begin{proof}
 In order to establish the weighted energy estimates~\eqref{equ:weighted_bound_u1}--\eqref{equ:weighted_bound_dx_dt_u1}, it is helpful to represent the component~$u_1(t)$ in Duhamel form as
 \begin{equation*}
  u_1(t) = \int_1^t U(t-s) \beta(x) u(s)^3 \, \ud s,
 \end{equation*}
 where we use the short-hand notation 
 \[
  U(t-s) := \frac{\sin((t-s)\jap{\nabla})}{\jap{\nabla}}.
 \]
 Then by finite speed of propagation for the Klein-Gordon equation we may write
 \begin{equation*}
  \chi_{\calS_t^R} u_1(t) = \chi_{\calS_t^R} \int_1^t U(t-s) \beta(x) \chi_{\calS_s^R} u(s)^3 \, \ud s.
 \end{equation*}
 Hence, the pointwise in time local decay estimate~\eqref{equ:weighted_decay} for the Klein-Gordon propagator and the decay assumption~\eqref{equ:growth_interior_energies_Linfty_assumption} imply uniformly for all $t \geq 1$ that
 \begin{align*}
  \bigl\| \jap{x}^{-2} \chi_{\calS_t^R} u_1(t) \bigr\|_{L^2_x(\bbR)} &\lesssim \int_1^t \bigl\| \jap{x}^{-2} U(t-s) \jap{x}^{-2} \bigr\|_{L^2_x \to L^2_x} \bigl\| \jap{x}^2 \beta(x) \bigr\|_{L^2_x} \bigl\| \chi_{\calS_s^R} u(s) \bigr\|_{L^\infty_x}^3 \, \ud s \\
  &\lesssim_{\beta(x)} \int_1^t \frac{1}{\jap{t-s}^{\frac{1}{2}}} \frac{(2 C_0 \varepsilon)^3}{\jap{s}^{\frac{3}{2}}} \, \ud s \\
  &\lesssim \frac{(2C_0\varepsilon)^3}{\jap{t}^{\frac{1}{2}}}.
 \end{align*}
 This gives the bound~\eqref{equ:weighted_bound_u1}. The proof of~\eqref{equ:weighted_bound_dx_u1} proceeds similarly. By finite speed of propagation we may write
 \begin{equation*}
  \chi_{\calS_t^R} \px u_1(t) = \chi_{\calS_t^R} \int_1^t \px U(t-s) \beta(x) \chi_{\calS_s^R} u(s)^3 \, \ud s.
 \end{equation*}
 Hence, by the pointwise in time local decay estimate~\eqref{equ:weighted_derivative_decay} and by~\eqref{equ:growth_interior_energies_Linfty_assumption}, we obtain for all $t \geq 1$ that
 \begin{align*}
  \bigl\| \jap{x}^{-2} \chi_{\calS_t^R} \px u_1(t) \bigr\|_{L^2_x} &\lesssim \int_1^t \bigl\| \jap{x}^{-2} \px U(t-s) \jap{x}^{-2} \bigr\|_{L^2_x \to L^2_x} \bigl\| \jap{x}^2 \beta(x) \bigr\|_{L^2_x} \bigl\| \chi_{\calS_s^R} u(s) \bigr\|_{L^\infty_x}^3 \, \ud s \\
  &\lesssim_{\beta(x)} \int_1^t \frac{1}{\jap{t-s}^{\frac{3}{2}}} \frac{(2 C_0 \varepsilon)^3}{\jap{s}^{\frac{3}{2}}} \, \ud s \\
  &\lesssim_{\beta(x)} \frac{(2C_0\varepsilon)^3}{\jap{t}^{\frac{3}{2}}}.
 \end{align*}

 For the proofs of the bounds~\eqref{equ:weighted_bound_dx2_u1} and~\eqref{equ:weighted_bound_dx_dt_u1} we first need to derive some auxiliary estimates. In view of the equation
 \begin{align*}
  (\Box + 1) (\pt u) = 3 (\beta_0 + \beta(x)) (\pt u) u^2,
 \end{align*}
 using finite speed of propagation and \eqref{equ:growth_interior_energies_Linfty_assumption} again, we conclude for all $t \geq 1$ that
 \begin{align*}
  \bigl\| \chi_{\calS_t^R} (\pt u)(t) \bigr\|_{L^2_x(\bbR)} &\lesssim \| (\pt u)(1) \|_{L^2_x(\bbR)} + \int_1^t \| \chi_{\calS_s^R} (\pt u)(s) \|_{L^2_x} \| \chi_{\calS_s^R} u(s)\|_{L^\infty_x}^2 \, \ud s \\
  &\lesssim \varepsilon + \int_1^t \frac{(2C_0\varepsilon)^2}{s} \| \chi_{\calS_s^R} (\pt u)(s) \|_{L^2_x} \, \ud s.
 \end{align*}
 By Gronwall's inequality it then follows for all $t \geq 1$ that
 \begin{equation} \label{equ:growth_bound_time_derivative}
  \bigl\| \chi_{\calS_t^R} (\pt u)(t) \bigr\|_{L^2_x(\bbR)} \lesssim \varepsilon t^{C (2C_0\varepsilon)^2},
 \end{equation}
 where $C > 0$ is some absolute constant. Similarly, from the equation
 \begin{align*}
  (\Box + 1) (Z u) = 3 (\beta_0 + \beta(x)) (Zu) u^2 + t (\partial_x \beta)(x) u^3,
 \end{align*}
 we infer for all $t \geq 1$ that
 \begin{align*}
  \bigl\| \chi_{\calS_t^R} (Z u)(t) \bigr\|_{L^2_x(\bbR)} &\lesssim \|Zu(1)\|_{L^2_x(\bbR)} + \int_1^t \bigl\| \chi_{\calS_s^R} Zu(s) \bigr\|_{L^2_x(\bbR)} \bigl\| \chi_{\calS_s^R} u(s) \bigr\|_{L^\infty_x(\bbR)}^2 \, \ud s \\
  &\quad \quad + \int_1^t s \| \partial_x \beta(x) \|_{L^2_x(\bbR)} \bigl\| \chi_{\calS_s^R} u(s) \bigr\|_{L^\infty_x(\bbR)}^3 \, \ud s \\
  &\lesssim \varepsilon + \int_1^t \frac{(2C_0\varepsilon)^2}{s} \bigl\| \chi_{\calS_s^R} Zu(s) \bigr\|_{L^2_x(\bbR)} \, \ud s + \int_1^t \frac{(2C_0\varepsilon)^3}{s^{\frac{1}{2}}} \, \ud s \\
  &\lesssim \varepsilon + (2C_0\varepsilon)^3 t^{\frac{1}{2}} + \int_1^t \frac{(2C_0\varepsilon)^2}{s} \bigl\| \chi_{\calS_s^R} Zu(s) \bigr\|_{L^2_x(\bbR)} \, \ud s.
 \end{align*}
 Hence, by Gronwall's inequality, we obtain for all $t \geq 1$ that 
 \begin{equation} \label{equ:growth_bound_bad_Z}
  \bigl\| \chi_{\calS_t^R} (Z u)(t) \bigr\|_{L^2_x(\bbR)} \lesssim \varepsilon t^{\frac{1}{2} + C (2C_0\varepsilon)^2}.
 \end{equation}

 We now turn to proving the bound~\eqref{equ:weighted_bound_dx2_u1}. To this end we will use that for any $t > 0$ we may write
 \begin{align*}
  \px u(t) = \frac{1}{t} t \px u = \frac{1}{t} \bigl( t \px u + x \pt u - x \pt u \bigr) = \frac{1}{t} \bigl( Zu - x \pt u \bigr).
 \end{align*}
 From the equation
 \begin{align*}
  \partial_x^2 u_1(t) = \int_1^t \px U(t-s) \bigl( (\px \beta)(x) u(s)^3 + 3 \beta(x) (\px u)(s) u(s)^2 \bigr) \, \ud s
 \end{align*}
 we then conclude for all $t \geq 1$ that
 \begin{align*}
  &\bigl\| \jap{x}^{-2} \chi_{\calS_t^R} \px^2 u_1(t) \bigr\|_{L^2_x(\bbR)} \\
  &\quad \lesssim \int_1^t \bigl\| \jap{x}^{-2} \px U(t-s) \jap{x}^{-2} \bigr\|_{L^2_x \to L^2_x} \bigl\| \jap{x}^2 (\px \beta)(x) \bigr\|_{L^2_x(\bbR)} \bigl\| \chi_{\calS_s^R} u(s) \bigr\|_{L^\infty_x(\bbR)}^3 \, \ud s \\
  &\quad \quad + \int_1^t \bigl\| \jap{x}^{-2} \px U(t-s) \jap{x}^{-2} \bigr\|_{L^2_x \to L^2_x} \bigl\| \jap{x}^2 \beta(x) \chi_{\calS_s^R} (\px u)(s) \bigr\|_{L^2_x(\bbR)} \bigl\| \chi_{\calS_s^R} u(s) \bigr\|_{L^\infty_x(\bbR)}^2 \, \ud s \\
  &\quad \lesssim \int_1^t \bigl\| \jap{x}^{-2} \px U(t-s) \jap{x}^{-2} \bigr\|_{L^2_x \to L^2_x} \bigl\| \jap{x}^2 (\px \beta)(x) \bigr\|_{L^2_x(\bbR)} \bigl\| \chi_{\calS_s^R} u(s) \bigr\|_{L^\infty_x(\bbR)}^3 \, \ud s \\
  &\quad \quad + \int_1^t \bigl\| \jap{x}^{-2} \px U(t-s) \jap{x}^{-2} \bigr\|_{L^2_x \to L^2_x} \frac{1}{s} \bigl\| \jap{x}^2 \beta(x) \bigr\|_{L^\infty_x(\bbR)} \bigl\| \chi_{\calS_s^R} (Z u)(s) \bigr\|_{L^2_x(\bbR)} \bigl\| \chi_{\calS_s^R} u(s) \bigr\|_{L^\infty_x(\bbR)}^2 \, \ud s \\
  &\quad \quad + \int_1^t \bigl\| \jap{x}^{-2} \px U(t-s) \jap{x}^{-2} \bigr\|_{L^2_x \to L^2_x} \frac{1}{s} \bigl\| \jap{x}^3 \beta(x) \bigr\|_{L^\infty_x(\bbR)} \bigl\| \chi_{\calS_s^R} (\pt u)(s) \bigr\|_{L^2_x(\bbR)} \bigl\| \chi_{\calS_s^R} u(s) \bigr\|_{L^\infty_x(\bbR)}^2 \, \ud s.
 \end{align*}
 Thus, by the pointwise in time local decay estimate~\eqref{equ:weighted_derivative_decay}, by~\eqref{equ:growth_interior_energies_Linfty_assumption}, and by the growth bounds~\eqref{equ:growth_bound_time_derivative}--\eqref{equ:growth_bound_bad_Z}, we obtain for all $t \geq 1$ that
 \begin{align*}
  \bigl\| \jap{x}^{-2} \chi_{\calS_t^R} \px^2 u_1(t) \bigr\|_{L^2_x(\bbR)} &\lesssim \int_1^t \frac{1}{\jap{t-s}^{\frac{3}{2}}} \frac{(2C_0\varepsilon)^3}{s^{\frac{3}{2}}} \, \ud s + \int_1^t \frac{1}{\jap{t-s}^{\frac{3}{2}}} \frac{1}{s} \varepsilon s^{\frac{1}{2} + C (2C_0\varepsilon)^2} \frac{(2C_0\varepsilon)^2}{s} \, \ud s \\
  &\quad \quad + \int_1^t \frac{1}{\jap{t-s}^{\frac{3}{2}}} \frac{1}{s} \varepsilon s^{C (2C_0\varepsilon)^2} \frac{(2C_0\varepsilon)^2}{s} \, \ud s \\
  &\lesssim (2C_0\varepsilon)^3 \int_1^t \frac{1}{\jap{t-s}^{\frac{3}{2}}} \frac{1}{s^{\frac{3}{2} - C (2C_0\varepsilon)^2}} \, \ud s \\
  &\lesssim \frac{ (2C_0\varepsilon)^3 }{t^{\frac{3}{2} - C (2C_0\varepsilon)^2}}.
 \end{align*}

 Finally, we establish the estimate~\eqref{equ:weighted_bound_dx_dt_u1}. Here we write
 \begin{align*}
  \px \pt u_1(t) = \int_1^t \px \cos( (t-s) \jap{\nabla} ) \beta(x) u(s)^3 \, \ud s.
 \end{align*}
 Proceeding similarly as in the derivation of the bound~\eqref{equ:weighted_bound_dx2_u1}, by finite speed of propagation, by the pointwise in time local decay estimate~\eqref{equ:weighted_derivative_decay2}, and by~\eqref{equ:growth_interior_energies_Linfty_assumption}, we compute for all $t \geq 1$ that
 \begin{align*}
  &\bigl\| \jap{x}^{-2} \chi_{\calS_t^R} \px \pt u_1(t) \bigr\|_{L^2_x(\bbR)} \\
  &\quad \lesssim \int_1^t \bigl\| \jap{x}^{-2} \px \cos( (t-s) \jap{\nabla} ) \jap{x}^{-2} \bigr\|_{H^1_x \to L^2_x} \bigl\| \jap{x}^2 \beta(x) \chi_{\calS_s^R} u(s)^3 \bigr\|_{H^1_x} \, \ud s \\
  &\quad \lesssim \int_1^t \frac{1}{\jap{t-s}^{\frac{3}{2}}} \frac{ (2C_0\varepsilon)^3 }{s^{\frac{3}{2} - C (2C_0\varepsilon)^2}} \, \ud s \\
  &\quad \lesssim \frac{ (2C_0\varepsilon)^3 }{t^{\frac{3}{2} - C (2C_0\varepsilon)^2}}.
 \end{align*}
 This finishes the proof of Lemma~\ref{lem:weighted_bounds_u1}.
\end{proof}

We now turn to the proof of Proposition~\ref{prop:interior_energy_estimate}.

\begin{proof}[Proof of Proposition~\ref{prop:interior_energy_estimate}]
We begin by computing that
\begin{align*}
 \frac{\ud}{\ud \rho} E_{int, \rho}(u) &= 2 \int_{H_\rho} \bigl( (\Box + 1) u \bigr) (\partial_t u) \frac{\rho}{t} \, \ud x \\
 &= 2 \int_{H_\rho} \bigl( \beta(x) u^3 + \beta_0 u^3 \bigr) (\partial_t u) \frac{\rho}{t} \, \ud x \\
 &\lesssim \|u\|_{L^\infty(H_\rho)}^2 \|u\|_{L^2(H_\rho)} \bigl\| (\partial_t u) {\textstyle \frac{\rho}{t}} \bigr\|_{L^2(H_\rho)} \\
 &\lesssim \|u\|_{L^\infty(H_\rho)}^2 E_{int, \rho}(u).
\end{align*}
Moreover, since $Z$ commutes with $\Box + 1$, we have
\begin{align*}
 \frac{\ud}{\ud \rho} E_{int, \rho}(Zu_0) &= 2 \int_{H_\rho} \bigl( (\Box + 1) (Z u_0) \bigr) (\partial_t Z u_0) \frac{\rho}{t} \, \ud x \\
 &= 6 \beta_0 \int_{H_\rho} (Zu) u^2 (\partial_t Z u_0) \frac{\rho}{t} \, \ud x \\
 &\lesssim \|u\|_{L^\infty(H_\rho)}^2 \bigl( \|Z u_0\|_{L^2(H_\rho)} + \|Z u_1\|_{L^2(H_\rho)} \bigr)  \bigl\| (\partial_t Z u_0) {\textstyle \frac{\rho}{t} } \bigr\|_{L^2(H_\rho)} \\
 &\lesssim \|u\|_{L^\infty(H_\rho)}^2 \bigl( E_{int, \rho}(Zu_0)^{\frac{1}{2}} + E_{int, \rho}(Zu_1)^{\frac{1}{2}} \bigr) E_{int, \rho}(Zu_0)^{\frac{1}{2}} \\
 &\lesssim \|u\|_{L^\infty(H_\rho)}^2 \calE_{int}(\rho).
\end{align*}
Thus, integrating in $\rho$ and invoking the decay assumption~\eqref{equ:growth_interior_energies_Linfty_assumption}, we have for any $1 \leq \tau \leq R$ that
\begin{equation*}
 E_{int, \tau}(u) + E_{int, \tau}(Zu_0) \leq \calE_{int}(1) + C \int_1^\tau \frac{(2C_0\varepsilon)^2}{\rho} \calE_{int}(\rho) \, \ud \rho.
\end{equation*}

Correspondingly, it remains to obtain control of the growth of $E_{int, \tau}(Zu_1)$, which is the heart of the matter. Here we compute
\[
 (\Box + 1) (Zu_1) = Z ( \beta(x) u^3 ) = 3 \beta(x) (Zu) u^2 + t (\partial_x \beta)(x) u^3
\]
and thus have to estimate for any $1 \leq \tau \leq R$,
\begin{equation} \label{equ:E_int_Z_u1}
 \begin{aligned}
  E_{int, \tau}(Zu_1) &= E_{int, 1}(Zu_1) + 2 \int_1^\tau \int_{H_\rho} \bigl( (\Box + 1)(Zu_1) \bigr) (\partial_t Zu_1) \frac{\rho}{t} \, \ud x \, \ud \rho \\
  &= E_{int, 1}(Zu_1) \\
  &\qquad + 6 \int_1^\tau \int_{H_\rho} \beta(x) (Zu) u^2 (\partial_t Zu_1) \frac{\rho}{t} \, \ud x \, \ud \rho + 2 \int_1^\tau \int_{H_\rho} t (\partial_x \beta)(x) u^3 (\partial_t Z u_1) \frac{\rho}{t} \, \ud x \, \ud \rho.
 \end{aligned}
\end{equation}
Using~\eqref{equ:growth_interior_energies_Linfty_assumption}, the second term on the right-hand side of~\eqref{equ:E_int_Z_u1} can be easily bounded by
\begin{align*}
 6 \int_1^\tau \int_{H_\rho} \beta(x) (Zu) u^2 (\partial_t Zu_1) \frac{\rho}{t} \, \ud x \, \ud \rho &\lesssim_\beta \int_1^\tau \|Zu\|_{L^2(H_\rho)} \|u\|_{L^\infty(H_\rho)}^2 \| (\partial_t Zu_1) {\textstyle \frac{\rho}{t}} \|_{L^2(H_\rho)} \, \ud \rho \\
 &\lesssim \int_1^\tau \frac{(2 C_0 \varepsilon)^2}{\rho} \calE_{int}(\rho) \, \ud \rho.
\end{align*}
In order to estimate the difficult third term on the right-hand side of~\eqref{equ:E_int_Z_u1}, we first note that by a change of variables
\begin{equation} \label{equ:E_int_Z_u1_difficult}
 \begin{aligned}
  \int_1^\tau \int_{H_\rho} t (\partial_x \beta)(x) u^3 (\partial_t Z u_1) \frac{\rho}{t} \, \ud x \, \ud \rho &= \int_1^\tau \int_{H_\rho} (\partial_x \beta)(x) u^3 (\partial_t Z u_1) \rho \, \ud x \, \ud \rho \\
  &= \int_1^\infty \int_{\bR} \chi_{\calS_t^\tau \backslash \calS_t^1} (\partial_x \beta)(x) u^3 (\partial_t Z u_1) \, t \, \ud x \, \ud t \\
  &\leq \int_1^\tau \int_{\bR} \chi_{\calS_t^\tau \backslash \calS_t^1} |(\partial_x \beta)(x)| |u|^3 |\partial_t Z u_1| \, t \, \ud x \, \ud t \\
  &\qquad + \int_\tau^\infty \int_{\bR} \chi_{\calS_t^\tau \backslash \calS_t^1} |(\partial_x \beta)(x)| |u|^3 |\partial_t Z u_1| \, t \, \ud x \, \ud t,
 \end{aligned}
\end{equation}
where $\chi_{\calS_t^\tau \backslash \calS_t^1}(x)$ is a sharp cut-off function given by
\begin{equation*}
 \chi_{\calS_t^\tau \backslash \calS_t^1}(x) = \left\{ \begin{array}{ll}
                                                                         1, \quad (t,x) \in \calS_t^\tau \backslash \calS_t^1, \\
                                                                         0, \quad \text{otherwise}.
                                                                        \end{array}
								\right.
\end{equation*}
Next, we decompose $\partial_t Zu_1$ into
\begin{equation} \label{equ:decompose_partial_t_Z_u}
 \begin{aligned}
  \partial_t Zu_1 &= \partial_t (t \partial_x + x \partial_t) u_1 \\
  &= t \partial_x \partial_t u_1 + \partial_x u_1 + x \partial_t^2 u_1 \\
  &= t \partial_x \partial_t u_1 + \partial_x u_1 + x \partial_x^2 u_1 - x u_1 + x \beta(x) u^3,
 \end{aligned}
\end{equation}
where in the last line we inserted the equation $(\partial_t^2 - \partial_x^2 + 1) u_1 = \beta(x) u^3$ for $\partial_t^2 u_1$. We now have to bound the two integrals on the right-hand side of~\eqref{equ:E_int_Z_u1_difficult} for each term in the above expression for $\partial_t Zu_1$.

We describe in detail how to handle the contribution of the most difficult term $t \partial_x \partial_t u_1$. For the first integral on the right-hand side of~\eqref{equ:E_int_Z_u1_difficult} we have, by the decay assumption~\eqref{equ:growth_interior_energies_Linfty_assumption} and the weighted energy estimate~\eqref{equ:weighted_bound_dx_dt_u1} for~$\partial_x \partial_t u_1$, for any $1 \leq \tau \leq R$ that
\begin{align*}
  &\int_1^\tau \int_{\bR} \chi_{\calS_t^\tau \backslash \calS_t^1} |(\partial_x \beta)(x)| |u|^3 |\partial_x \partial_t u_1| \, t^2 \, \ud x \, \ud t \\
  &\leq \int_1^\tau \bigl\| \jap{x}^2 \partial_x \beta \bigr\|_{L^2_x(\bbR)} \bigl\| \chi_{\calS_t^R} u(t) \bigr\|_{L^\infty_x(\bbR)}^3 \bigl\| \jap{x}^{-2} \chi_{\calS_t^R} \partial_x \partial_t u_1 \bigr\|_{L^2_x(\bR)} t^2 \, \ud t \\
  &\lesssim_\beta \int_1^\tau \frac{(2 C_0 \varepsilon)^3}{t^{\frac{3}{2}}} \frac{(2 C_0 \varepsilon)^3}{t^{\frac{3}{2} - C (2 C_0 \varepsilon)^2}} t^2 \, \ud t \\
  &\lesssim (2 C_0 \varepsilon)^6 \int_1^\tau \frac{1}{t^{1 - C(2C_0\varepsilon)^2}} \, \ud t \\
  &\lesssim (2 C_0 \varepsilon)^6 \tau^{C (2C_0\varepsilon)^2}.
\end{align*}
For the second integral on the right-hand side of~\eqref{equ:E_int_Z_u1_difficult} we have to proceed more carefully and exploit the decreasing size of the spatial support of $\chi_{\calS_t^\tau \backslash \calS_t^1}$ for large $t$ in order to obtain an integrable time decay at infinity. More precisely, for any $1 \leq \tau \leq R$ and for $p = \infty-$, we have by the decay assumption~\eqref{equ:growth_interior_energies_Linfty_assumption} and the weighted energy estimate~\eqref{equ:weighted_bound_dx_dt_u1} that
\begin{align*}
 &\int_\tau^\infty \int_{\bR} \chi_{\calS_t^\tau \backslash \calS_t^1} |(\partial_x \beta)(x)| |u|^3 |\partial_x \partial_t u_1| \, t^2 \, \ud x \, \ud t \\
 &\lesssim \int_\tau^\infty \bigl\| \chi_{\calS_t^\tau \backslash \calS_t^1} \bigr\|_{L^p_x(\bbR)} \bigl\| \jap{x}^2 \partial_x \beta \bigr\|_{L^{\frac{2p}{p-2}}_x(\bbR)} \bigl\| \chi_{\calS_t^R} u(t) \bigr\|_{L^\infty_x(\bbR)}^3 \bigl\| \jap{x}^{-2} \chi_{\calS_t^R} \partial_x \partial_t u_1(t) \bigr\|_{L^2_x(\bR)} t^2 \, \ud t \\
 &\lesssim_\beta \int_\tau^\infty \biggl( \frac{\tau^2-1}{\sqrt{t^2-1}} \biggr)^{\frac{1}{p}} \frac{(2 C_0 \varepsilon)^3}{t^{\frac{3}{2}}} \frac{(2 C_0 \varepsilon)^3}{t^{\frac{3}{2} - C (2 C_0 \varepsilon)^2}} t^2 \, \ud t \\
 &\lesssim (2 C_0 \varepsilon)^6 \tau^{\frac{2}{p}} \int_\tau^\infty \frac{1}{t^{1 - C (2 C_0 \varepsilon)^2 + \frac{1}{p}}} \, \ud t \\
 &\lesssim (2 C_0 \varepsilon)^6 \tau^{C (2 C_0 \varepsilon)^2 + \frac{1}{p}}
\end{align*}
as long as $\frac{1}{p} > C (2 C_0 \varepsilon)^2$. Here we used that for any fixed $t \geq \tau$ there holds
\begin{align*}
 \bigl\| \chi_{\calS_t^\tau \backslash \calS_t^1} \bigr\|_{L^p_x(\bbR)} &= \biggl( 2 \bigl( \sqrt{t^2-1} - \sqrt{t^2 - \tau^2} \bigr) \biggr)^{\frac{1}{p}} = \biggl( 2 \frac{\tau^2 - 1}{\sqrt{t^2-1} + \sqrt{t^2-\tau^2}} \biggr)^{\frac{1}{p}} \lesssim \biggl( \frac{\tau^2-1}{\sqrt{t^2-1}} \biggr)^{\frac{1}{p}}.
\end{align*}
For the purposes of the proof of Proposition~\ref{prop:interior_energy_estimate} we pick $\frac{1}{p} := 2 C (2C_0\varepsilon)^2$.

Now we are left to estimate the two integrals on the right-hand side of~\eqref{equ:E_int_Z_u1_difficult} for the remaining terms in the decomposition~\eqref{equ:decompose_partial_t_Z_u} of $\partial_t Zu_1$. Their contributions can be estimated similarly as above, using the weighted energy estimates~\eqref{equ:weighted_bound_u1}--\eqref{equ:weighted_bound_dx2_u1}, the decay assumption~\eqref{equ:growth_interior_energies_Linfty_assumption}, and by absorbing the weight~$x$ by the spatial decay of $(\px \beta)(x)$.

Putting all of the above estimates together yields that uniformly for all $1 \leq \tau \leq R$ there holds
\begin{equation} \label{equ:growth_interior_energies}
 \calE_{int}(\tau) \lesssim \calE_{int}(1) + (2 C_0 \varepsilon)^6 \tau^{3 C (2 C_0 \varepsilon)^2} + \int_1^{\tau} \frac{(2 C_0 \varepsilon)^2}{\rho} \calE_{int}(\tau) \, \ud \rho.
\end{equation}
Since by the exterior energy estimate~\eqref{equ:exterior_energy_estimate} from Proposition~\ref{prop:exterior_energy_estimate} we have
\[
 \calE_{int}(1) \lesssim \limsup_{T \to \infty} \calE_{ext}(T) \lesssim \calE_{ext}(1) \lesssim \varepsilon^2,
\]
the desired slow growth estimate~\eqref{equ:interior_E_int_goal} for the interior energy functional $\calE_{int}$ now follows from~\eqref{equ:growth_interior_energies} by Gronwall's inequality. 
\end{proof}

In the next corollary we record several key growth estimates for the variable $w = \rho^{\frac{1}{2}} \cosh^{\frac{1}{2}}(y) u$ in hyperbolic coordinates. These are immediate consequences of Proposition~\ref{prop:interior_energy_estimate}.

\begin{corollary} \label{cor:interior_growth_estimates_hyperbolic_coordinates}
 Let $R \geq 1$ and suppose that the assumptions of Proposition~\ref{prop:interior_energy_estimate} are in place. Then there exists a constant $C_1 > 0$, independent of $C_0$ and $R$, such that
 \begin{align}
  \sup_{1 \leq \rho \leq R} \, \rho^{-\delta} \| w(\rho, \cdot)\|_{H^1_y(\bbR)} &\leq C_1 \varepsilon, \label{equ:H1y_norm_w} \\ 
  \sup_{1 \leq \rho \leq R} \, \rho^{-\delta} \| (\partial_\rho w)(\rho, \cdot) \|_{L^2_y(\bbR)} &\leq C_1 \varepsilon, \label{equ:L2_norm_partial_rho_w} \\
  \sup_{1 \leq \rho \leq R} \, \rho^{-\delta} \bigl\| \cosh^{-1}(\cdot) (\partial_\rho w)(\rho, \cdot) \bigr\|_{H^1_y(\bbR)} &\leq C_1 \varepsilon. \label{equ:H1y_cosh_norm_partial_rho_w}
 \end{align}
\end{corollary}
\begin{proof}
 It is straightforward to see that the bounds~\eqref{equ:H1y_norm_w} and~\eqref{equ:H1y_cosh_norm_partial_rho_w} are consequences of the slow growth estimate~\eqref{equ:interior_E_int_goal} for the interior energy functional $\calE_{int}(\rho)$ established in the previous Proposition~\ref{prop:interior_energy_estimate}. For the proof of~\eqref{equ:L2_norm_partial_rho_w} we first observe that in the interior region $\jap{x} \leq t$, since $\rho^2 = t^2 - x^2 = (t+x)(t-x)$, we may write
 \begin{align*}
  \partial_\rho w = \rho^{-1} (t \pt + x \px) w &= \rho^{-1} \bigl( (t \pt + x \px)w \pm (t \px + x \pt) w \mp (t \px + x \pt) w \bigr) \\
  &= \rho^{-1} \bigl( (t \pm x) \pt w + (x \pm t) \px w \mp \partial_y w \bigr) \\
  &= \rho^{-1} \biggl( \frac{\rho^2}{t \mp x} \pt w - \frac{\rho^2}{x \mp t} \px w \mp \partial_y w \biggr) \\
  &= \frac{\rho}{t \mp x} \pt w - \frac{\rho}{x \mp t} \px w \mp \rho^{-1} \partial_y w.
 \end{align*}
 It follows that in the interior region there holds 
 \begin{align*}
  |\partial_\rho w| \leq \frac{\rho}{t} |\pt w| + \frac{\rho}{t} |\px w| + \rho^{-1} |\partial_y w|.
 \end{align*}
 Hence, we conclude that
 \begin{align*}
  \| (\partial_\rho w)(\rho, \cdot) \|_{L^2_y(\bbR)} \lesssim \| {\textstyle \frac{\rho}{t}} (\pt w)(\rho, \cdot) \|_{L^2_y(\bbR)} + \bigl\| {\textstyle \frac{\rho}{t}} (\px w)(\rho, \cdot) \|_{L^2_y(\bbR)} + \| (\partial_y w)(\rho, \cdot) \|_{L^2_y(\bbR)} \lesssim \calE_{int}(\rho)^{\frac{1}{2}},
 \end{align*}
 which finishes the proof.
\end{proof}

\subsection{Interior Decay Estimate}

\begin{proof}[Proof of Proposition~\ref{prop:interior_decay}]
 Let $R \geq 1$. We are now in the position to carry out the bootstrap argument~\eqref{equ:interior_bootstrap_assumption}--\eqref{equ:interior_bootstrap_conclusion}, which reads in terms of the variable $w = \rho^{\frac{1}{2}} \cosh^{\frac{1}{2}}(y) u$ in hyperbolic coordinates
 \begin{align}
  \sup_{1 \leq \rho \leq R} \, \sup_{y \in \bbR} \, |w(\rho, y)| &\leq 2 C_0 \varepsilon \label{equ:interior_bootstrap_assumption_hyperbolic_coord} \\
  &\qquad \qquad \Longrightarrow \quad \sup_{1 \leq \rho \leq R} \, \sup_{y \in \bbR} \, |w(\rho, y)| \leq C_0 \varepsilon. \label{equ:interior_bootstrap_conclusion_hyperbolic_coord}
 \end{align}
 In view of our initial data assumptions~\eqref{equ:data_assumption_interior_decay_prop}, the solution~$u(t)$ to~\eqref{equ:nlkg} obeys the exterior decay estimate~\eqref{equ:exterior_decay_estimate}. Choosing $C_0 \gg 1$ sufficiently large and combining~\eqref{equ:exterior_decay_estimate} with the bootstrap hypothesis~\eqref{equ:interior_bootstrap_assumption}, we may thus assume that the solution~$u(t)$ satisfies the decay estimate
 \begin{equation*}
  \sup_{t \geq 1} \, t^{\frac{1}{2}} \bigl\| \chi_{\calS_t^R} u(t) \bigr\|_{L^\infty_x(\bbR)} \leq 2 C_0 \varepsilon.
 \end{equation*}
 Hence, we can invoke the interior energy estimates from Proposition~\ref{prop:interior_energy_estimate} and Corollary~\ref{cor:interior_growth_estimates_hyperbolic_coordinates}, where importantly the constant $C_1 > 0$ is independent of $C_0$ and $R$.

 \medskip

 Equipped with the slow growth estimate~\eqref{equ:H1y_norm_w} for the $H^1_y$-norm of the variable $w$ and the slow growth estimate~\eqref{equ:L2_norm_partial_rho_w} for the $L^2_y$-norm of $\partial_\rho w$, we now begin with the proof of~\eqref{equ:interior_bootstrap_conclusion_hyperbolic_coord}. 
 To this end we decompose the variable $w$ into a low-frequency and a high-frequency component, using a $\rho$-dependent frequency cut-off,
 \begin{equation*}
  w(\rho) = (P_{\leq \rho^{\sigma}} w)(\rho) + (P_{> \rho^\sigma} w)(\rho),
 \end{equation*}
 where the absolute constant $0 < \sigma \ll 1$ is chosen such that $0 < \delta \ll \sigma \ll 1$.

 \medskip

 For the high-frequency component it is straightforward to derive a uniform $L^\infty_y$-estimate. Indeed, by Bernstein estimates and the slow growth bound~\eqref{equ:H1y_norm_w} for the $H^1_y$-norm of $w$, we obtain 
 \begin{equation} \label{equ:Linfty_high_freq_bound}
  \| P_{> \rho^{\sigma}} w \|_{L^\infty_y(\bbR)} \lesssim \rho^{-\frac{\sigma}{2}} \| \partial_y w(\rho) \|_{L^2_y(\bbR)} \lesssim \rho^{-\frac{\sigma}{2}} \rho^{\delta} C_1 \varepsilon \leq \frac{C_0}{2} \varepsilon
 \end{equation}
 as long as $C_0 \gg 1$ is chosen sufficiently large initially.

 \medskip

 In order to bound the low-frequency component, we first show that $P_{\leq \rho^\sigma} w$ essentially solves an ODE, for which we will then invoke a suitable $L^\infty$-estimate. Rearranging the nonlinear Klein-Gordon equation~\eqref{equ:nlkg_hyperbolic_coord} written in hyperbolic coordinates, we have 
 \begin{equation} \label{equ:nlkg_w}
  \begin{aligned}
   \bigl( \partial_\rho^2 + 1 \bigr) w &= \frac{1}{\rho} \calB w^3 + \frac{1}{\rho^2} \partial_y^2 w - \frac{1}{\rho^2} \tanh(y) \partial_y w - \frac{3}{4 \rho^2} \cosh^{-2}(y) w,  
  \end{aligned}
 \end{equation}
 Here we use the short-hand notation 
 \begin{equation} \label{equ:definition_B}
  \calB(\rho, y) := \frac{\beta_0}{\cosh(y)} + \frac{\beta(\rho \sinh(y))}{\cosh(y)}
 \end{equation}
 and note that the following bounds for the coefficient $\calB$ can be readily verified
 \begin{equation*}
  \sup_{\rho \geq 1} \, \| \calB(\rho, \cdot) \|_{L^\infty_y(\bbR)} \lesssim 1, \quad \sup_{\rho \geq 1} \, \rho \| \partial_\rho \calB(\rho, \cdot) \|_{L^\infty_y(\bbR)} \lesssim 1.
 \end{equation*}
 Next we apply the low frequency projection $P_{\leq \rho^\sigma}$ to~\eqref{equ:nlkg_w} to find that
 \begin{equation} \label{equ:nlkg_w_low_freq_proj}
  \begin{aligned}
   \bigl( \partial_\rho^2 + 1 \bigr) (P_{\leq \rho^\sigma} w) &=  \frac{1}{\rho} P_{\leq \rho^\sigma} \bigl( \calB w^3 \bigr) - [ P_{\leq \rho^\sigma}, \partial_\rho^2 ] w + \frac{1}{\rho^2} \partial_y^2 (P_{\leq \rho^\sigma} w) \\
   &\quad \quad - \frac{1}{\rho^2} P_{\leq \rho^\sigma} \bigl( \tanh(y) \partial_y w \bigr) - \frac{3}{4 \rho^2} P_{\leq \rho^\sigma} \bigl( \cosh^{-2}(y) w \bigr).  
  \end{aligned}
 \end{equation}
 At this point all terms on the right-hand side of~\eqref{equ:nlkg_w_low_freq_proj} apart from the first two can already be easily treated as integrable error terms thanks to the slowly growing energy bound~\eqref{equ:H1y_norm_w}. We further decompose the first term on the right-hand side into remainder terms and a main term that will have to be incorporated into the ODE $L^\infty$-estimate,
 \begin{equation*} 
  \begin{aligned}
   \frac{1}{\rho} P_{\leq \rho^\sigma} \bigl( \calB w^3 \bigr) &= \frac{1}{\rho} P_{\leq \rho^\sigma} \Bigl( \calB (P_{\leq \rho^\sigma} w)^3 \Bigr) + \frac{1}{\rho} P_{\leq \rho^\sigma} \Bigl( \calB \bigl( w^3 - (P_{\leq \rho^\sigma} w)^3 \bigr) \Bigr)  \\
   &= \frac{1}{\rho} (P_{\leq 2^{10} \rho^\sigma} \calB) (P_{\leq \rho^\sigma} w)^3 - \frac{1}{\rho} P_{>\rho^\sigma} \Bigl( (P_{\leq 2^{10} \rho^\sigma} \calB) (P_{\leq \rho^\sigma} w)^3 \Bigr) \\
   &\quad \quad + \frac{1}{\rho} P_{\leq \rho^\sigma} \Bigl( \calB \bigl( w^3 - (P_{\leq \rho^\sigma} w)^3 \bigr) \Bigr),
  \end{aligned}
 \end{equation*}
 where we used that by frequency considerations there holds
 \begin{equation*}
  P_{\leq \rho^\sigma} \Bigl( (P_{> 2^{10} \rho^\sigma} \calB) (P_{\leq \rho^\sigma} w)^3 \Bigr) = 0.
 \end{equation*}
 Thus, we arrive at the following equation for the low-frequency component
 \begin{equation} \label{equ:nlkg_w_low_freq_proj_rearranged}
  \begin{aligned}
   \bigl( \partial_\rho^2 + 1 \bigr) (P_{\leq \rho^\sigma} w) &= \frac{1}{\rho} (P_{\leq 2^{10} \rho^\sigma} \calB) (P_{\leq \rho^\sigma} w)^3 - [ P_{\leq \rho^\sigma}, \partial_\rho^2 ] w + \calR,
  \end{aligned}
 \end{equation}
 with an integrable remainder term $\calR$ given by
 \begin{equation} \label{equ:definition_R}
  \begin{aligned}
   \calR &= \frac{1}{\rho^2} \partial_y^2 (P_{\leq \rho^\sigma} w) - \frac{1}{\rho^2} P_{\leq \rho^\sigma} \bigl( \tanh(y) \partial_y w \bigr) - \frac{3}{4 \rho^2} P_{\leq \rho^\sigma} \bigl( \cosh^{-2}(y) w \bigr) \\ 
   &\quad \quad + \frac{1}{\rho} P_{\leq \rho^\sigma} \Bigl( \calB \bigl( w^3 - (P_{\leq \rho^\sigma} w)^3 \bigr) \Bigr) - \frac{1}{\rho} P_{>\rho^\sigma} \Bigl( (P_{\leq 2^{10} \rho^\sigma} \calB) (P_{\leq \rho^\sigma} w)^3 \Bigr).
  \end{aligned}
 \end{equation}

Next, we view~\eqref{equ:nlkg_w_low_freq_proj_rearranged} as an ODE in the variable~$\rho$ and invoke a corresponding $L^\infty$-estimate by multiplying~\eqref{equ:nlkg_w_low_freq_proj_rearranged} by $2 \partial_\rho (P_{\leq \rho^\sigma} w)$, which gives
\begin{equation} \label{equ:nlkg_w_low_freq_proj_ODE_form_intermediate}
 \begin{aligned}
  \partial_\rho \Bigl( (P_{\leq \rho^\sigma} w)^2 + \bigl( \partial_\rho (P_{\leq \rho^\sigma} w) \bigr)^2 \Bigr) &= \frac{2}{\rho} \partial_\rho (P_{\leq \rho^\sigma} w) (P_{\leq 2^{10} \rho^\sigma} \calB) (P_{\leq \rho^\sigma} w)^3  \\
  &\quad \quad - 2 \partial_\rho (P_{\leq \rho^\sigma} w) [ P_{\leq \rho^\sigma}, \partial_\rho^2 ] w + 2 \partial_\rho (P_{\leq \rho^\sigma} w) \calR.
 \end{aligned}
\end{equation}
Then the first term on the right-hand side of~\eqref{equ:nlkg_w_low_freq_proj_ODE_form_intermediate} can be rearranged as
\begin{equation} \label{equ:nlkg_w_low_freq_proj_ODE_form_term1_rearranged}
 \begin{aligned}
  &\frac{2}{\rho} (P_{\leq 2^{10} \rho^\sigma} \calB) (P_{\leq \rho^\sigma} w)^3 \partial_\rho (P_{\leq \rho^\sigma} w) \\ 
  &\quad = \partial_\rho \Bigl( \frac{1}{2 \rho} (P_{\leq 2^{10} \rho^\sigma} \calB) (P_{\leq \rho^\sigma} w)^4 \Bigr) - \frac{1}{2\rho} \partial_\rho (P_{\leq 2^{10} \rho^\sigma} \calB) (P_{\leq \rho^\sigma} w)^4 + \frac{1}{2\rho^2} (P_{\leq 2^{10} \rho^\sigma} \calB) (P_{\leq \rho^\sigma} w)^4.
 \end{aligned}
\end{equation}
The second term on the right-hand side of~\eqref{equ:nlkg_w_low_freq_proj_ODE_form_intermediate} involves the commutator $[ P_{\leq \rho^\sigma}, \partial_\rho^2 ]$, which is given by
 \begin{equation*}
  [ P_{\leq \rho^\sigma}, \partial_\rho^2 ] w = - 2 \partial_\rho \bigl( (\partial_\rho P_{\leq \rho^\sigma}) w \bigr) + (\partial_\rho^2 P_{\leq \rho^\sigma}) w = \partial_\rho \biggl( \frac{2}{\rho} (P_{\sim \rho^\sigma}' w) \biggr) + \frac{1}{\rho^2} P_{\sim \rho^\sigma}'' w.
 \end{equation*}
 Here $P_{\sim \rho^\sigma}'$ and $P_{\sim \rho^\sigma}''$ are Littlewood-Paley projections with standard $L^1$ kernels that localize to frequencies $|\eta| \sim \rho^\sigma$ and that are defined by
 \begin{align*}
  P_{\sim \rho^\sigma}' w &= \frac{\sigma}{\sqrt{2\pi}} \int_{\bbR} e^{+ i \eta y} \frac{\eta}{\rho^\sigma} \varphi'\Bigl( \frac{\eta}{\rho^\sigma} \Bigr) \bigl( \calF_y w \bigr)(\eta) \, \ud \eta, \\
  P_{\sim \rho^\sigma}'' w &= \frac{\sigma}{\sqrt{2\pi}} \int_{\bbR} e^{+ i \eta y} \frac{\eta}{\rho^\sigma} \varphi' \Bigl( \frac{\eta}{\rho^\sigma} \Bigr) \bigl( \calF_y w \bigr)(\eta) \, \ud \eta \\
  &\quad + \frac{\sigma^2}{\sqrt{2\pi}} \int_{\bbR} e^{+ i \eta y} \biggl( \frac{\eta}{\rho^\sigma} \varphi'\Bigl( \frac{\eta}{\rho^\sigma} \Bigr) + \Bigl( \frac{\eta}{\rho^\sigma} \Bigr)^2 \varphi''\Bigl( \frac{\eta}{\rho^\sigma} \Bigr) \biggr) \bigl( \calF_y w \bigr)(\eta) \, \ud \eta.
 \end{align*}
 Thus, we may rewrite the second term on the right-hand side of~\eqref{equ:nlkg_w_low_freq_proj_ODE_form_intermediate} as
 \begin{equation} \label{equ:nlkg_w_low_freq_proj_ODE_form_term2_rearranged}
  \begin{aligned}
   &-2 \partial_\rho(P_{\leq \rho^\sigma} w) [P_{\leq \rho^\sigma}, \partial_\rho^2] w \\
   &\quad \quad = -2 \partial_\rho(P_{\leq \rho^\sigma} w) \partial_\rho \biggl( \frac{2}{\rho} (P_{\sim \rho^\sigma}' w) \biggr) - 2 \partial_\rho(P_{\leq \rho^\sigma} w) \frac{1}{\rho^2} (P_{\sim \rho^\sigma}'' w) \\
   &\quad \quad = - \partial_\rho \biggl( \frac{4}{\rho} \partial_\rho(P_{\leq \rho^\sigma} w) (P_{\sim \rho^\sigma}' w) \biggr) + \frac{4}{\rho} \partial_\rho^2 (P_{\leq \rho^\sigma} w) (P_{\sim \rho^\sigma}' w) - \frac{2}{\rho^2} \partial_\rho(P_{\leq \rho^\sigma} w) (P_{\sim \rho^\sigma}'' w).
  \end{aligned}
 \end{equation}
 Inserting~\eqref{equ:nlkg_w_low_freq_proj_ODE_form_term1_rearranged} and \eqref{equ:nlkg_w_low_freq_proj_ODE_form_term2_rearranged} into~\eqref{equ:nlkg_w_low_freq_proj_ODE_form_intermediate}, we obtain that
 \begin{equation} \label{equ:nlkg_w_low_freq_proj_ODE_form}
  \begin{aligned}
   &\partial_\rho \Bigl( (P_{\leq \rho^\sigma} w)^2 + \bigl( \partial_\rho (P_{\leq \rho^\sigma} w) \bigr)^2 \Bigr) \\
   &\quad \quad= \partial_\rho \Bigl( \frac{1}{2 \rho} (P_{\leq 2^{10} \rho^\sigma} \calB) (P_{\leq \rho^\sigma} w)^4 \Bigr) - \partial_\rho \biggl( \frac{4}{\rho} \partial_\rho(P_{\leq \rho^\sigma} w) (P_{\sim \rho^\sigma}' w) \biggr) \\
   &\quad \quad \quad \quad + 2 \partial_\rho (P_{\leq \rho^\sigma} w) \calR - \frac{1}{2\rho} \partial_\rho (P_{\leq 2^{10} \rho^\sigma} \calB) (P_{\leq \rho^\sigma} w)^4 + \frac{1}{2\rho^2} (P_{\leq 2^{10} \rho^\sigma} \calB) (P_{\leq \rho^\sigma} w)^4 \\
   &\quad \quad \quad \quad + \frac{4}{\rho} \partial_\rho^2 (P_{\leq \rho^\sigma} w) (P_{\sim \rho^\sigma}' w) - \frac{2}{\rho^2} \partial_\rho(P_{\leq \rho^\sigma} w) (P_{\sim \rho^\sigma}'' w).
  \end{aligned}
 \end{equation}
 Next, we introduce the quantity
 \begin{equation*}
  M(\rho, y) := (P_{\leq \rho^\sigma} w)^2(\rho, y) + \bigl( \partial_\rho (P_{\leq \rho^\sigma} w) \bigr)^2(\rho, y).
 \end{equation*}
 Integrating~\eqref{equ:nlkg_w_low_freq_proj_ODE_form} in $\rho$ from $\rho = 1$ to $\rho = \tau$ for any $1 \leq \tau \leq R$ and taking the $L^\infty_y$-norm, we find that
 \begin{equation} \label{equ:Linfty_applied_to_ODE}
  \begin{aligned}
   \| M(\tau) \|_{L^\infty_y} &\lesssim \|M(1)\|_{L^\infty_y} + \sup_{1 \leq \rho \leq \tau} \, \frac{1}{\rho}  \bigl\| (P_{\leq 2^{10} \rho^\sigma} \calB) (P_{\leq \rho^\sigma} w)^4 \bigr\|_{L^\infty_y} + \sup_{1 \leq \rho \leq \tau} \, \frac{1}{\rho} \bigl\| \partial_\rho(P_{\leq \rho^\sigma} w) (P_{\sim \rho^\sigma}' w) \bigr\|_{L^\infty_y} \\
   &\qquad + \int_1^\tau \bigl\| \partial_\rho (P_{\leq \rho^\sigma} w) \calR \bigr\|_{L^\infty_y} \, \ud \rho + \int_1^\tau \frac{1}{\rho} \bigl\| \partial_\rho (P_{\leq 2^{10} \rho^\sigma} \calB) (P_{\leq \rho^\sigma} w)^4 \bigr\|_{L^\infty_y} \, \ud \rho \\
   &\qquad + \int_1^\tau \frac{1}{\rho^2} \bigl\| (P_{\leq 2^{10} \rho^\sigma} \calB) (P_{\leq \rho^\sigma} w)^4 \bigr\|_{L^\infty_y} \, \ud \rho + \int_1^\tau \frac{1}{\rho} \bigl\| \partial_\rho^2 (P_{\leq \rho^\sigma} w) (P_{\sim \rho^\sigma}' w) \bigr\|_{L^\infty_y} \, \ud \rho \\
   &\qquad + \int_1^\tau \frac{1}{\rho^2} \bigl\| \partial_\rho(P_{\leq \rho^\sigma} w) (P_{\sim \rho^\sigma}'' w) \bigr\|_{L^\infty_y} \, \ud \rho \\
   &\equiv I + II + \ldots + VIII.
  \end{aligned}
 \end{equation}
It now remains to estimate all terms on the right-hand side of~\eqref{equ:Linfty_applied_to_ODE}, for which the slow growth estimates~\eqref{equ:H1y_norm_w} and~\eqref{equ:L2_norm_partial_rho_w} crucially enter. We also recall the choice $0 < \delta \ll \sigma \ll 1$ of the small constants.

\medskip

\noindent {\it Term $I$.} We first observe that
 \begin{align*}
  \|M(\rho)\|_{L^\infty_y} &\lesssim \bigl\| P_{\leq \rho^\sigma} w \bigr\|_{L^\infty_y}^2 + \frac{1}{\rho^2} \bigl\| P_{\sim \rho^\sigma}' w \bigr\|_{L^\infty_y}^2 + \bigl\| P_{\leq \rho^\sigma} (\partial_\rho w) \bigr\|_{L^\infty_y}^2 \\
  &\lesssim \rho^\sigma \|w(\rho)\|_{L^2_y}^2 + \rho^{-2+\sigma} \|w(\rho)\|_{L^2_y}^2 + \rho^\sigma \| (\partial_\rho w)(\rho) \|_{L^2_y}^2 \\
  &\lesssim \rho^\sigma \bigl( \|w(\rho)\|_{L^2_y}^2 + \| (\partial_\rho w)(\rho) \|_{L^2_y}^2 \bigr).
 \end{align*}
 By the slow growth estimates~\eqref{equ:H1y_norm_w} and~\eqref{equ:L2_norm_partial_rho_w} it then follows that
 \begin{equation*}
  \|M(\rho)\|_{L^\infty_y} \lesssim \rho^{\sigma + 2 \delta} \varepsilon^2.
 \end{equation*}
 In particular, we have the bound
 \begin{equation*}
  \|M(1)\|_{L^\infty_y} \lesssim \varepsilon^2.
 \end{equation*}
 
\medskip 

\noindent {\it Term $II$.} By Sobolev embedding and the slow growth estimate~\eqref{equ:H1y_norm_w} we conclude that
\begin{align*}
 \bigl\| (P_{\leq 2^{10} \rho^\sigma} \calB) (P_{\leq \rho^\sigma} w)^4 \bigr\|_{L^\infty_y} &\lesssim \| \calB \|_{L^\infty_y} \|w(\rho)\|_{L^\infty_y}^4 \lesssim \| \calB \|_{L^\infty_y} \|w(\rho)\|_{H^1_y}^4 \lesssim \rho^{+4\delta} \varepsilon^4.
\end{align*}
Thus, we obtain that
\begin{align*}
 \sup_{1 \leq \rho \leq \tau} \, \frac{1}{\rho}  \bigl\| (P_{\leq 2^{10} \rho^\sigma} \calB) (P_{\leq \rho^\sigma} w)^4 \bigr\|_{L^\infty_y} \lesssim \sup_{1 \leq \rho \leq \tau} \, \rho^{-1 + 4\delta} \varepsilon^4 \lesssim \varepsilon^4.
\end{align*}

\medskip 

\noindent {\it Term $III$.} Using Young's inequality and~\eqref{equ:H1y_norm_w}, we have for any $\mu > 0$ that
\begin{align*}
 \sup_{1 \leq \rho \leq \tau} \, \frac{1}{\rho} \bigl\| \partial_\rho(P_{\leq \rho^\sigma} w) (P_{\sim \rho^\sigma}' w) \bigr\|_{L^\infty_y} &\lesssim \mu \sup_{1 \leq \rho \leq \tau} \, \bigl\| \partial_\rho(P_{\leq \rho^\sigma} w) \bigr\|_{L^\infty_y}^2 + \frac{1}{\mu} \sup_{1 \leq \rho \leq \tau} \, \rho^{-2} \| (P_{\sim \rho^\sigma}' w) \|_{L^\infty_y}^2 \\
  &\lesssim \mu \sup_{1 \leq \rho \leq \tau} \, \|M(\rho)\|_{L^\infty_y} + \frac{1}{\mu} \varepsilon^2.
\end{align*}
For sufficiently small $\mu > 0$, we will be able to absorb the first term into the left-hand side of the final estimate.

\medskip 

\noindent {\it Term $IV$.} First, by Young's inequality we find for any $\mu > 0$ that
\begin{align*}
 \int_1^\tau \bigl\| \partial_\rho (P_{\leq \rho^\sigma} w) \calR \bigr\|_{L^\infty_y} \, \ud \rho &\leq \Bigl( \sup_{1 \leq \rho \leq \tau} \bigl\| \partial_\rho (P_{\leq \rho^\sigma} w) \bigr\|_{L^\infty_y} \Bigr) \int_1^\tau \| \calR(\rho) \|_{L^\infty_y} \, \ud \rho \\
  &\leq \Bigl( \sup_{1 \leq \rho \leq \tau} \, \|M(\rho)\|_{L^\infty_y}^{\frac{1}{2}} \Bigr) \int_1^\tau \| \calR(\rho) \|_{L^\infty_y} \, \ud \rho \\
  &\leq \mu \sup_{1 \leq \rho \leq \tau} \, \|M(\rho)\|_{L^\infty_y} + \frac{1}{\mu} \biggl( \int_1^\tau \| \calR(\rho) \|_{L^\infty_y} \, \ud \rho \biggr)^2.
\end{align*}
We will absorb the first term into the left-hand side of the final estimate. For the remainder term $\calR$ it therefore suffices to show that
\begin{equation*}
 \int_1^R \|\calR(\rho)\|_{L^\infty_y(\bbR)} \, \ud \rho \lesssim \varepsilon.
\end{equation*}
Using the slow growth estimate~\eqref{equ:H1y_norm_w} and standard Littlewood-Paley estimates, this is straightforward to verify for the first four components in the definition~\eqref{equ:definition_R} of $\calR$. For the second to last component of~\eqref{equ:definition_R} we use~\eqref{equ:H1y_norm_w} and the identity 
\[
 w^3 - (P_{\leq \rho^\sigma} w)^3 = (P_{>\rho^\sigma} w) \bigl( w^2 + w (P_{\leq \rho^\sigma} w) + (P_{\leq \rho^\sigma} w)^2 \bigr) 
\]
to estimate 
\begin{align*}
 \int_1^R \frac{1}{\rho} \Bigl\| P_{\leq \rho^\sigma} \Bigl( \calB \bigl( w^3 - (P_{\leq \rho^\sigma} w)^3 \bigr) \Bigr) \Bigr\|_{L^\infty_y} \, \ud \rho &\lesssim \int_1^R \frac{1}{\rho} \| \calB \|_{L^\infty_y} \| P_{>\rho^\sigma} w \|_{L^\infty_y} \|w\|_{L^\infty_y}^2 \, \ud \rho \\
 &\lesssim \int_1^R \rho^{-1 - \frac{\sigma}{2}} \|w\|_{H^1_y}^3 \, \ud \rho \\
 &\lesssim \int_1^R \rho^{-1 - \frac{\sigma}{2} + 3 \delta} \varepsilon^3 \, \ud \rho \\
 &\lesssim \varepsilon^3.
\end{align*}
Finally, we decompose the last component of~\eqref{equ:definition_R} into
\begin{equation} \label{equ:last_component_definition_R_decomposed}
 \begin{aligned}
  P_{>\rho^\sigma} \Bigl( (P_{\leq 2^{10} \rho^\sigma} \calB) (P_{\leq \rho^\sigma} w)^3 \Bigr) &= P_{>\rho^\sigma} \Bigl( (P_{[2^{-10} \rho^\sigma, 2^{10} \rho^\sigma]} \calB) (P_{\leq \rho^\sigma} w)^3 \Bigr) + P_{>\rho^\sigma} \Bigl( (P_{\leq 2^{-10} \rho^\sigma} \calB) (P_{\leq \rho^\sigma} w)^3 \Bigr).
 \end{aligned}
\end{equation}
By frequency considerations the second term on the right-hand side can be written as
\begin{equation*}
 P_{>\rho^\sigma} \Bigl( (P_{\leq 2^{-10} \rho^\sigma} \calB) (P_{\leq \rho^\sigma} w)^3 \Bigr) = P_{>\rho^\sigma} \Bigl( (P_{\leq 2^{-10} \rho^\sigma} \calB) \bigl( (P_{\leq \rho^\sigma} w)^3 - (P_{\leq 2^{-10} \rho^\sigma} w)^3 \bigr) \Bigr)
\end{equation*}
and its contribution can therefore be estimated analogously to the contribution of the second to last component of~\eqref{equ:definition_R} above. The contribution of the first term on the right-hand side of~\eqref{equ:last_component_definition_R_decomposed} can be bounded using the slow growth estimate~\eqref{equ:H1y_norm_w} together with the following additional decay thanks to~\eqref{equ:Pk_beta_variable_Linfty} and~\eqref{equ:Pk_beta_const_Linfty}, 
\begin{align*}
 \| P_{[2^{-10} \rho^\sigma, 2^{10} \rho^\sigma]} \calB \|_{L^\infty_y} &\lesssim \sum_{ 2^k \sim \rho^\sigma } \Bigl\| P_k \Bigl( \frac{\beta( \rho \sinh(\cdot) )}{\cosh(\cdot)} \Bigr) \Bigr\|_{L^\infty_y} + \sum_{ 2^k \sim \rho^\sigma } \Bigl\| P_k \Bigl( \frac{\beta_0}{\cosh(\cdot)} \Bigr) \Bigr\|_{L^\infty_y} \lesssim \rho^{-1 + \sigma} + \rho^{-\sigma}.
\end{align*}

\medskip 

\noindent {\it Term $V$.} Using the slow growth estimate~\eqref{equ:H1y_norm_w} and the fact that
\begin{equation*}
 \| \partial_\rho (P_{\leq 2^{10} \rho^\sigma} \calB) \|_{L^\infty_y} \lesssim \rho^{-1} \| P_{\sim 2^{10} \rho^\sigma}' \calB \|_{L^\infty_y} + \| \partial_\rho \calB \|_{L^\infty_y} \lesssim \rho^{-1},
\end{equation*}
it is straightforward to verify that uniformly for all $1 \leq \tau \leq R$, 
\begin{equation*}
 \int_1^\tau \frac{1}{\rho} \bigl\| \partial_\rho (P_{\leq 2^{10} \rho^\sigma} \calB) (P_{\leq \rho^\sigma} w)^4 \bigr\|_{L^\infty_y} \, \ud \rho \lesssim \varepsilon^4.
\end{equation*}

\medskip 

\noindent {\it Term $VI$.} Similarly, using the slow growth estimate~\eqref{equ:H1y_norm_w} one easily obtains that uniformly for all $1 \leq \tau \leq R$, 
\begin{equation*}
 \int_1^\tau \frac{1}{\rho^2} \bigl\| (P_{\leq 2^{10} \rho^\sigma} \calB) (P_{\leq \rho^\sigma} w)^4 \bigr\|_{L^\infty_y} \, \ud \rho \lesssim \varepsilon^4.
\end{equation*}

\medskip 

\noindent {\it Term $VII$.} We begin by inserting the ODE~\eqref{equ:nlkg_w_low_freq_proj_rearranged} for $\partial_\rho^2 (P_{\leq \rho^\sigma} w)$. Then using the slow growth estimate~\eqref{equ:H1y_norm_w} for the $H^1_y$-norm of $w$ as well as the slow growth estimate~\eqref{equ:L2_norm_partial_rho_w} for the $L^2_y$-norm of $\partial_\rho w$, by Sobolev embedding and standard Littlewood-Paley estimates, we arrive at the overall estimate
\begin{equation*}
 \int_1^\tau \frac{1}{\rho} \bigl\| \partial_\rho^2 (P_{\leq \rho^\sigma} w) (P_{\sim \rho^\sigma}' w) \bigr\|_{L^\infty_y} \, \ud \rho \lesssim \varepsilon^2.
\end{equation*}
Here it is convenient to insert the commutator term $[P_{\leq \rho^\sigma}, \partial_\rho^2] w$ in the form 
\[
 [P_{\leq \rho^\sigma}, \partial_\rho^2] w = - (\partial_\rho^2 P_{\leq \rho^\sigma}) w - 2 (\partial_\rho P_{\leq \rho^\sigma}) \partial_\rho w = - \frac{1}{\rho^2} (P_{\sim \rho^\sigma}'' w) + \frac{2}{\rho} (P_{\sim \rho^\sigma}' (\partial_\rho w))
\]
and to invoke the slow growth estimate~\eqref{equ:L2_norm_partial_rho_w} for the $L^2_y$-norm of $\partial_\rho w$ to bound 
\[
 \|P_{\sim \rho^\sigma}' (\partial_\rho w)\|_{L^\infty_y} \lesssim \rho^{\frac{\sigma}{2}} \|\partial_\rho w\|_{L^2_y} \lesssim \rho^{\frac{\sigma}{2} + \delta} \varepsilon.
\]

\medskip 

\noindent {\it Term $VIII$.} Finally, it is straightforward to obtain a suitable estimate on the last term $VIII$ using Young's inequality and~\eqref{equ:H1y_norm_w}.

\medskip 

Putting all of the above estimates together, we obtain for any $1 \leq \tau \leq R$ and any $\mu > 0$ that
\begin{align*}
 \| M(\tau) \|_{L^\infty_y} \lesssim \varepsilon^2 + \mu \sup_{1 \leq \rho \leq \tau} \| M(\rho) \|_{L^\infty_y} + \frac{1}{\mu} \varepsilon^2.
\end{align*}
Hence, upon taking the supremum in $\tau$ over $1 \leq \tau \leq R$ and choosing $\mu > 0$ sufficiently small, we conclude
\begin{equation} \label{equ:M_uniform_bound}
 \sup_{1 \leq \tau \leq R} \, \|M(\tau)\|_{L^\infty_y} \lesssim \varepsilon^2,
\end{equation}
with an implicit constant independent of $C_0$ and $1 \leq \tau \leq R$. In particular, we arrive at the desired conclusion
\begin{equation} \label{equ:Linfty_low_freq_bound}
 \sup_{1 \leq \rho \leq R} \, \bigl\| (P_{\leq \rho^\sigma} w)(\rho, \cdot) \bigr\|_{L^\infty_y(\bbR)} \leq \frac{C_0}{2} \varepsilon,
\end{equation}
as long as $C_0$ is chosen sufficiently large initially. Combining this low-frequency estimate~\eqref{equ:Linfty_low_freq_bound} with the high-frequency estimate~\eqref{equ:Linfty_high_freq_bound} we obtain the desired bound
\[
 \sup_{1 \leq \rho \leq R} \, \sup_{y \in \bbR} \, |w(\rho, y)| \leq C_0 \varepsilon,
\]
which proves~\eqref{equ:interior_bootstrap_conclusion_hyperbolic_coord} and closes the bootstrap argument~\eqref{equ:interior_bootstrap_assumption_hyperbolic_coord}--\eqref{equ:interior_bootstrap_conclusion_hyperbolic_coord}.
\end{proof}

\begin{remark}
 We would like to point out that the contribution of the commutator term $[P_{\leq \rho^\sigma}, \partial_\rho^2] w$ in the proof of Proposition~\ref{prop:interior_decay} would be much easier to estimate if we had a slow growth estimate of the form $$\sup_{1 \leq \rho \leq R} \rho^{-\delta} \| (\partial_y \partial_\rho w)(\rho) \|_{L^2_y(\bbR)} \lesssim \varepsilon$$ at our disposal. However, it seems difficult to establish such a bound since the energy functional $E_{int, \rho}$ only provides good control of a weighted $L^2_y$-norm of derivatives in $\rho$ with an unfavorable weight.
\end{remark}

\begin{remark} \label{rem:convergence_L2}
 The proof of Proposition~\ref{prop:interior_decay} establishes something more than the interior decay estimate~\eqref{equ:interior_decay_estimate}. Namely, that there exists a function $b \in L^\infty_y(\bbR)$ such that for every fixed $y \in \bbR$ the following limit exists
 \begin{equation} \label{equ:definition_amplitude_b}
  b(y) := \lim_{\rho \to \infty} \, \Bigl( \bigl( \partial_\rho (P_{\leq \rho^\sigma} w)(\rho, y) \bigr)^2 + \bigl( (P_{\leq \rho^\sigma} w)(\rho, y) \bigr)^2 \Bigr)^{\frac{1}{2}}.
 \end{equation}
 Moreover, a small modification of the proof of Proposition~\ref{prop:interior_decay} shows that the limit \eqref{equ:definition_amplitude_b} exists in $L^2_y(\bbR)$ as well. Hence, we have that $b \in L^\infty_y(\bbR) \cap L^2_y(\bbR)$. 
\end{remark}

\begin{remark} \label{rem:amplitude_limit}
 The definition of the function $b(y)$ in~\eqref{equ:definition_amplitude_b} is independent of the frequency cut-off $P_{\leq \rho^\sigma}$. In fact, we have for every fixed $y \in \bbR$ that 
 \begin{equation*}
  b(y) = \lim_{\rho \to \infty} \Bigl( (\partial_\rho w)(\rho, y)^2 + w(\rho, y)^2 \Bigr)^{\frac{1}{2}}.
 \end{equation*}
\end{remark}
\begin{proof}
 In view of the definition~\eqref{equ:definition_amplitude_b} of $b(y)$, it suffices to prove that for every fixed $y \in \bbR$ we have
 \[
  \lim_{\rho \to \infty} \Bigl( \bigl| \partial_\rho (P_{>\rho^\sigma} w)(\rho, y) \bigr| + \bigl| (P_{> \rho^\sigma} w)(\rho, y) \bigr| \Bigr) = 0.
 \]
 This is immediate for the high-frequency component $P_{>\rho^\sigma} w$ of the variable $w$ by Bernstein estimates and~\eqref{equ:H1y_norm_w}. The conclusion for $\partial_\rho (P_{>\rho^\sigma} w)$ requires a more careful argument though when the derivative in $\rho$ falls on $w$. In that case we exploit that for every fixed $y \in \bbR$ with $|y| \leq L$, we may bound
 \[
  \bigl| (P_{> \rho^\sigma} \partial_\rho w)(\rho, y) \bigr| \lesssim_{L} \bigl\| \cosh^{-1}(\cdot) (P_{> \rho^\sigma} \partial_\rho w)(\rho, \cdot) \bigr\|_{L^\infty_y}
 \]
 with an implicit constant that of course depends on the size of $y$. Using the following standard commutator estimate for Littlewood-Paley projections,
 \begin{equation*}
  \bigl\| [ f, P_k ] \bigr\|_{L^2_y \to L^\infty_y} \lesssim 2^{-\frac{1}{2} k} \| f' \|_{L^\infty_y}, \quad k \in \bbZ,
 \end{equation*}
 we obtain by the slow growth estimates~\eqref{equ:L2_norm_partial_rho_w} and~\eqref{equ:H1y_cosh_norm_partial_rho_w} from Corollary~\ref{cor:interior_growth_estimates_hyperbolic_coordinates} that
 \begin{align*}
  &\bigl\| \cosh^{-1}(\cdot) (P_{> \rho^\sigma} \partial_\rho w)(\rho, \cdot) \bigr\|_{L^\infty_y} \\
  &\quad \lesssim \bigl\| P_{> \rho^\sigma} \bigl( \cosh^{-1}(\cdot) (\partial_\rho w)(\rho, \cdot) \bigr) \bigr\|_{L^\infty_y} + \sum_{k = \lfloor \log_2(\rho^\sigma) \rfloor}^\infty \bigl\| [ \cosh^{-1}(\cdot), P_{k} ] (\partial_\rho w)(\rho, \cdot) \bigr\|_{L^\infty_y} \\
  &\quad \lesssim \rho^{-\frac{\sigma}{2}} \bigl\| \cosh^{-1}(\cdot) (\partial_\rho w)(\rho, \cdot) \bigr\|_{H^1_y} + \sum_{k = \lfloor \log_2(\rho^\sigma) \rfloor}^\infty 2^{-\frac{1}{2} k} \| (\partial_\rho w)(\rho, \cdot) \|_{L^2_y} \\
  &\quad \lesssim \rho^{-\frac{\sigma}{2}} \Bigl( \bigl\| \cosh^{-1}(\cdot) (\partial_\rho w)(\rho, \cdot) \bigr\|_{H^1_y} + \| (\partial_\rho w)(\rho, \cdot) \|_{L^2_y} \Bigr) \\
  &\quad \lesssim \rho^{-\frac{\sigma}{2}} \rho^{+ \delta} \varepsilon,
 \end{align*}
 which vanishes as $\rho \to \infty$. This completes the proof.
\end{proof}

\section{Asymptotic Behavior} \label{sec:asymptotics}

In this final section we determine the asymptotic behavior of solutions to~\eqref{equ:nlkg} in the interior region.

\begin{proof}[Proof of Corollary~\ref{cor:asymptotics}]
The asymptotic behavior of $w(\rho, y) = (P_{\leq \rho^\sigma} w)(\rho, y) + (P_{> \rho^\sigma} w)(\rho, y)$ as $\rho \to \infty$ is determined by the behavior of the low-frequency component $P_{\leq \rho^\sigma} w$, because the high-frequency component $P_{> \rho^\sigma} w$ enjoys the additional decay
\[
 \bigl\| P_{> \rho^\sigma} w \bigr\|_{L^\infty_y} \lesssim \rho^{-\frac{\sigma}{2}} \| \partial_y w \|_{L^2_y} \lesssim \rho^{-\frac{\sigma}{2} + \delta} \varepsilon \to 0 \quad \text{ as } \rho \to \infty.
\]

We recall from~\eqref{equ:nlkg_w_low_freq_proj_rearranged} that the low-frequency component $P_{\leq \rho^\sigma} w$ satisfies the equation
\begin{align*}
 \bigl( \partial_\rho^2 + 1 \bigr) (P_{\leq \rho^\sigma} w) &= \frac{1}{\rho} (P_{\leq 2^{10} \rho^\sigma} \calB) (P_{\leq \rho^\sigma} w)^3 - [ P_{\leq \rho^\sigma}, \partial_\rho^2 ] w + \calR,
\end{align*}
where $\calB$ and $\calR$ are defined as in~\eqref{equ:definition_B}, respectively as in~\eqref{equ:definition_R}. In order to uncover the asymptotic behavior of the low-frequency component $P_{\leq \rho^\sigma} w$, we need to take into account the oscillatory behavior of the solution. To this end we introduce the variables
\begin{equation}
 W_{\pm} := e^{\mp i \rho} \bigl( \partial_\rho( P_{\leq \rho^\sigma} w) \pm i (P_{\leq \rho^\sigma} w) \bigr).
\end{equation}
Then we have
\begin{equation*}
 P_{\leq \rho^\sigma} w = \frac{1}{2i} \bigl( e^{+i\rho} W_{+} - e^{-i\rho} W_{-} \bigr) = \pm \Im \, \bigl( e^{\pm i \rho} W_{\pm} \bigr),
\end{equation*}
and it holds that
\begin{align*}
 W_{-} = \overline{W_{+}}, \quad \quad |W_{+}|^2 = |W_{-}|^2 = W_{-} W_{+} = (\partial_\rho (P_{\leq \rho^\sigma} w))^2 + (P_{\leq \rho^\sigma} w)^2.
\end{align*}
The proof of Proposition~\ref{prop:interior_decay}, specifically the estimate~\eqref{equ:M_uniform_bound}, implies that $|W_{\pm}|$ is uniformly bounded. Moreover, by Remark~\ref{rem:convergence_L2} there exists a non-negative function $b \in L^\infty_y(\bbR) \cap L^2_y(\bbR)$ such that for every fixed $y \in \bbR$ there holds
\begin{equation*}
 b(y)^2 = \lim_{\rho \to \infty} \, |W_{\pm}(\rho, y)|^2.
\end{equation*}

\medskip

Next, a computation shows that $W_{\pm}$ satisfies
\begin{equation} \label{equ:partial_rho_W_pm}
 \partial_\rho W_{\pm} = e^{\mp i \rho} (\partial_\rho^2 + 1) (P_{\leq \rho^\sigma} w) = e^{\mp i \rho}  \frac{1}{\rho} (P_{\leq 2^{10} \rho^\sigma} \calB) (P_{\leq \rho^\sigma} w)^3 - e^{\mp i \rho} [ P_{\leq \rho^\sigma}, \partial_\rho^2 ] w + e^{\mp i \rho} \calR.
\end{equation}
Since $W_{-} = \overline{W_{+}}$, in what follows we only consider the equation for $\partial_\rho W_{+}$. Inserting the expansion
\begin{align*}
 (P_{\leq \rho^\sigma} w)^3 &= -\frac{1}{8i} \bigl( e^{+i\rho} W_{+} - e^{-i\rho} W_{-} \bigr)^3 \\
 &= -\frac{1}{8i} \bigl( e^{+3i\rho} W_{+}^3 - 3 e^{i \rho} (W_+)^2 W_{-} + 3 e^{-i\rho} W_{+} (W_{-})^2 - e^{-3i\rho} (W_{-})^3 \bigr)
\end{align*}
into~\eqref{equ:partial_rho_W_pm}, we find that
\begin{equation} \label{equ:partial_rho_W_plus}
 \begin{aligned}
  \partial_\rho W_{+} &=  - i \frac{3}{8} \frac{1}{\rho} (P_{\leq 2^{10} \rho^\sigma} \calB) |W_{+}|^2 W_{+} \\
  &\qquad + \frac{i}{8} \frac{1}{\rho} e^{-i\rho} (P_{\leq 2^{10} \rho^\sigma} \calB) \bigl( e^{+3i\rho} W_{+}^3 + 3 e^{-i\rho} W_{+} (W_{-})^2 - e^{-3i\rho} (W_{-})^3 \bigr) \\
  &\qquad  - e^{- i \rho} [ P_{\leq \rho^\sigma}, \partial_\rho^2 ] w + e^{- i \rho} \calR.
 \end{aligned}
\end{equation}
Then the second term on the right-hand side of~\eqref{equ:partial_rho_W_plus} is better behaved asymptotically due to the oscillating phase factors, while the third and fourth terms are remainder terms with better decay properties. However, the first term on the right-hand side contains a resonant part, which we carefully peel off now. We write
\begin{align*}
 P_{\leq 2^{10} \rho^\sigma} \calB &= P_{\leq 2^{10} \rho^\sigma} \biggl( \frac{\beta_0 + \beta(\rho \sinh(\cdot))}{\cosh(\cdot)} \biggr) = \frac{\beta_0}{\cosh(y)} - P_{>2^{10} \rho^\sigma} \biggl( \frac{\beta_0}{\cosh(\cdot)} \biggr) + P_{\leq 2^{10} \rho^\sigma} \biggl( \frac{\beta(\rho \sinh(\cdot))}{\cosh(\cdot)} \biggr)
\end{align*}
and note that the last two terms have additional decay in $\rho$. Indeed, by~\eqref{equ:Pk_beta_const_Linfty} it holds that
\begin{equation*}
 \biggl\| P_{>2^{10} \rho^\sigma} \biggl( \frac{\beta_0}{\cosh(\cdot)} \biggr) \biggr\|_{L^\infty_y} \lesssim \sum_{k \gtrsim \log_2(\rho^\sigma)} \biggl\| P_k \biggl( \frac{\beta_0}{\cosh(\cdot)} \biggr) \biggr\|_{L^\infty_y} \lesssim \sum_{k \gtrsim \log_2(\rho^\sigma)} 2^{-k} \lesssim \rho^{-\sigma},
\end{equation*}
and in view of~\eqref{equ:Pk_beta_variable_Linfty} we have 
\begin{equation*}
 \biggl\| P_{\leq 2^{10} \rho^\sigma} \biggl( \frac{\beta(\rho \sinh(\cdot))}{\cosh(\cdot)} \biggr) \biggr\|_{L^\infty_y} \lesssim \sum_{k \lesssim \log_2(\rho^\sigma)} \biggl\| P_k \biggl( \frac{\beta(\rho \sinh(\cdot))}{\cosh(\cdot)} \biggr) \biggr\|_{L^\infty_y} \lesssim \sum_{k \lesssim \log_2(\rho^\sigma)} \frac{2^k}{\rho} \lesssim \rho^{-1+\sigma}.
\end{equation*}

Returning to the equation for $\partial_\rho W_{+}$, we may therefore write
\begin{equation*}
 \partial_\rho W_{+} = -i \frac{3 \beta_0}{8} \frac{1}{\rho \cosh(y)} |W_{+}|^2 W_{+} + \partial_\rho S_{+} + T_{+},
\end{equation*}
or equivalently,
\begin{equation} \label{equ:partial_rho_W_minus_S}
 \partial_\rho \bigl( W_{+} - S_{+} \bigr) = -i  \frac{3 \beta_0}{8} \frac{1}{\rho \cosh(y)} |W_{+}|^2 W_{+} + T_{+},
\end{equation}
where
\begin{align*}
 S_{+} &:= \frac{1}{8 \rho} (P_{\leq 2^{10} \rho^\sigma} \calB) \Bigl( \frac{1}{2} e^{+2i\rho} W_+^3 - \frac{3}{2} e^{-2i\rho} W_+ (W_-)^2 + \frac{1}{4} e^{-4i\rho} (W_-)^3 \Bigr) - \frac{2}{\rho}  e^{-i\rho} (P_{\sim \rho^\sigma}' w)
\end{align*}
satisfies $|S_+| \lesssim \frac{1}{\rho}$ and $T_+$ consists of all other (integrable) remainder terms. In view of the above observations, the proof of Proposition~\ref{prop:interior_decay}, and Remark~\ref{rem:convergence_L2}, there exists a small constant $0 < \nu \ll 1$ such that we have 
\begin{equation*}
 |T_+| \lesssim \rho^{-1-\nu}, \quad \bigl| b^2 - |W_+|^2 \bigr| = \calO( \rho^{-\nu} ).
\end{equation*}
Then multiplying~\eqref{equ:partial_rho_W_minus_S} by the integrating factor
\begin{align*}
 \exp \Bigl( + i \frac{3 \beta_0}{8} \frac{b(y)^2}{\cosh(y)} \log(\rho) \Bigr), 
\end{align*}
we may conclude that 
\begin{align*}
 &\partial_\rho \biggl( (W_+ - S_+) \exp \Bigl( + i \frac{3 \beta_0}{8} \frac{b(y)^2}{\cosh(y)} \log(\rho) \Bigr) \biggr) \\
 &\quad \quad = \biggl( + i \frac{3}{8} \frac{\beta_0}{\rho \cosh(y)} \bigl( b(y)^2 - |W_+|^2 \bigr) W_+ + T_+ \biggr) \exp \Bigl( + i \frac{3 \beta_0}{8} \frac{b(y)^2}{\cosh(y)} \log(\rho) \Bigr) \\
 &\quad \quad \equiv \calO(\rho^{-1-\nu}).
\end{align*}
Hence, for every fixed $y \in \bbR$ the following limit exists
\begin{equation*}
 a(y) := \lim_{\rho \to \infty} \, W_{+}(\rho, y) \exp \Bigl( + i \frac{3 \beta_0}{8} \frac{b(y)^2}{\cosh(y)} \log(\rho) \Bigr)
\end{equation*}
and it is clear that 
\begin{equation*}
 |a(y)| = b(y).
\end{equation*}
Thus, we have
\begin{equation*}
 \Bigl| a(y) - W_{+}(\rho, y) \exp \Bigl( + i \frac{3 \beta_0}{8} \frac{b(y)^2}{\cosh(y)} \log(\rho) \Bigr) \Bigr| = \calO(\rho^{-\nu}).
\end{equation*}
It follows that
\begin{align*}
 w(\rho, y) &= (P_{\leq \rho^\sigma} w)(\rho, y) + (P_{> \rho^\sigma} w)(\rho, y) \\
 &= \Im \bigl( e^{i\rho} W_{+}(\rho, y) \bigr) + (P_{> \rho^\sigma} w)(\rho, y) \\
 &= \Im \Bigl( e^{i ( \rho - \frac{3}{8} \beta_0 \frac{|a(y)|^2}{\cosh(y)} \log(\rho) )} a(y) \Bigr) + \calO(\rho^{-\nu}),
\end{align*}
which gives the asymptotic behavior~\eqref{equ:main_asymptotics} for the solution to~\eqref{equ:nlkg} in the interior region.
\end{proof}

\bibliographystyle{amsplain}
\bibliography{references}

\providecommand{\bysame}{\leavevmode\hbox to3em{\hrulefill}\thinspace}
\providecommand{\MR}{\relax\ifhmode\unskip\space\fi MR }
\providecommand{\MRhref}[2]{%
  \href{http://www.ams.org/mathscinet-getitem?mr=#1}{#2}
}
\providecommand{\href}[2]{#2}
\begin{thebibliography}{10}

\bibitem{CL18}
T.~Candy and H.~Lindblad, \emph{Long range scattering for the cubic {D}irac
  equation on {$\mathbb{R}^{1+1}$}}, Differential Integral Equations
  \textbf{31} (2018), no.~7-8, 507--518.

\bibitem{ChenPus19}
G.~Chen and F.~Pusateri, \emph{The 1d nonlinear {S}chr\"odinger equation with a
  weighted {L1} potential}, Preprint arXiv:1912.10949.

\bibitem{Del16}
J.-M. Delort, \emph{Modified scattering for odd solutions of cubic nonlinear
  {S}chr\"odinger equations with potential in dimension one}, Preprint
  hal-01396705.

\bibitem{Del01}
\bysame, \emph{Existence globale et comportement asymptotique pour l'\'equation
  de {K}lein-{G}ordon quasi lin\'eaire \`a donn\'ees petites en dimension 1},
  Ann. Sci. \'Ecole Norm. Sup. (4) \textbf{34} (2001), no.~1, 1--61.

\bibitem{Del06}
\bysame, \emph{Erratum: ``{G}lobal existence and asymptotic behavior for the
  quasilinear {K}lein-{G}ordon equation with small data in dimension 1''
  ({F}rench) [{A}nn. {S}ci. \'{E}cole {N}orm. {S}up. (4) {\bf 34} (2001), no.
  1, 1--61; mr1833089]}, Ann. Sci. \'{E}cole Norm. Sup. (4) \textbf{39} (2006),
  no.~2, 335--345.

\bibitem{GLS16}
V.~Georgescu, M.~Larenas, and A.~Soffer, \emph{Abstract theory of pointwise
  decay with applications to wave and {S}chr\"{o}dinger equations}, Ann. Henri
  Poincar\'{e} \textbf{17} (2016), no.~8, 2075--2101.

\bibitem{Ger08}
C.~G\'{e}rard, \emph{A proof of the abstract limiting absorption principle by
  energy estimates}, J. Funct. Anal. \textbf{254} (2008), no.~11, 2707--2724.

\bibitem{GermPusRou18}
P.~Germain, F.~Pusateri, and F.~Rousset, \emph{The nonlinear {S}chr\"{o}dinger
  equation with a potential}, Ann. Inst. H. Poincar\'{e} Anal. Non Lin\'{e}aire
  \textbf{35} (2018), no.~6, 1477--1530.

\bibitem{HN98}
N.~Hayashi and P.~Naumkin, \emph{Asymptotics for large time of solutions to the
  nonlinear {S}chr\"{o}dinger and {H}artree equations}, Amer. J. Math.
  \textbf{120} (1998), no.~2, 369--389.

\bibitem{HN08}
\bysame, \emph{The initial value problem for the cubic nonlinear
  {K}lein-{G}ordon equation}, Z. Angew. Math. Phys. \textbf{59} (2008), no.~6,
  1002--1028.

\bibitem{HN12}
\bysame, \emph{Quadratic nonlinear {K}lein-{G}ordon equation in one dimension},
  J. Math. Phys. \textbf{53} (2012), no.~10, 103711, 36.

\bibitem{H97}
L.~H\"ormander, \emph{Lectures on nonlinear hyperbolic differential equations},
  Math\'ematiques \& Applications (Berlin) [Mathematics \& Applications],
  vol.~26, Springer-Verlag, Berlin, 1997.

\bibitem{HSS99}
W.~Hunziker, I.~M. Sigal, and A.~Soffer, \emph{Minimal escape velocities},
  Comm. Partial Differential Equations \textbf{24} (1999), no.~11-12,
  2279--2295.

\bibitem{IT15}
M.~Ifrim and D.~Tataru, \emph{Global bounds for the cubic nonlinear
  {S}chr\"{o}dinger equation ({NLS}) in one space dimension}, Nonlinearity
  \textbf{28} (2015), no.~8, 2661--2675.

\bibitem{Jensen80}
A.~Jensen, \emph{Spectral properties of {S}chr\"{o}dinger operators and
  time-decay of the wave functions results in {$L^{2}({\bf R}^{m})$}, {$m\geq
  5$}}, Duke Math. J. \textbf{47} (1980), no.~1, 57--80.

\bibitem{Jensen84}
\bysame, \emph{Spectral properties of {S}chr\"{o}dinger operators and
  time-decay of the wave functions. {R}esults in {$L^{2}({\bf R}^{4})$}}, J.
  Math. Anal. Appl. \textbf{101} (1984), no.~2, 397--422.

\bibitem{KJ79}
A.~Jensen and T.~Kato, \emph{Spectral properties of {S}chr\"{o}dinger operators
  and time-decay of the wave functions}, Duke Math. J. \textbf{46} (1979),
  no.~3, 583--611.

\bibitem{JMP84}
A.~Jensen, \'{E}. Mourre, and P.~Perry, \emph{Multiple commutator estimates and
  resolvent smoothness in quantum scattering theory}, Ann. Inst. H.
  Poincar\'{e} Phys. Th\'{e}or. \textbf{41} (1984), no.~2, 207--225.

\bibitem{JSS91}
J.-L. Journ\'{e}, A.~Soffer, and C.~D. Sogge, \emph{Decay estimates for
  {S}chr\"{o}dinger operators}, Comm. Pure Appl. Math. \textbf{44} (1991),
  no.~5, 573--604.

\bibitem{KatPus11}
J.~Kato and F.~Pusateri, \emph{A new proof of long-range scattering for
  critical nonlinear {S}chr\"{o}dinger equations}, Differential Integral
  Equations \textbf{24} (2011), no.~9-10, 923--940.

\bibitem{Kl80}
S.~Klainerman, \emph{Global existence for nonlinear wave equations}, Comm. Pure
  Appl. Math. \textbf{33} (1980), no.~1, 43--101.

\bibitem{Kl85}
\bysame, \emph{Global existence of small amplitude solutions to nonlinear
  {K}lein-{G}ordon equations in four space-time dimensions}, Comm. Pure Appl.
  Math. \textbf{38} (1985), no.~5, 631--641.

\bibitem{Kl93}
\bysame, \emph{Remark on the asymptotic behavior of the {K}lein-{G}ordon
  equation in {$\mathbb{R}^{n+1}$}}, Comm. Pure Appl. Math. \textbf{46} (1993),
  no.~2, 137--144.

\bibitem{KMM17}
M.~Kowalczyk, Y.~Martel, and C.~Mu\~{n}oz, \emph{Kink dynamics in the
  {$\phi^4$} model: asymptotic stability for odd perturbations in the energy
  space}, J. Amer. Math. Soc. \textbf{30} (2017), no.~3, 769--798.

\bibitem{KMM17_1}
\bysame, \emph{On asymptotic stability of nonlinear waves}, S\'{e}minaire
  {L}aurent {S}chwartz---\'{E}quations aux d\'{e}riv\'{e}es partielles et
  applications. {A}nn\'{e}e 2016--2017, Ed. \'{E}c. Polytech., Palaiseau, 2017,
  pp.~Exp. No. XVIII, 27.

\bibitem{LarS15}
M.~Larenas and A.~Soffer, \emph{Abstract theory of decay estimates: perturbed
  {H}amiltonians}, Preprint arXiv:1508.04490.

\bibitem{Lind90}
H.~Lindblad, \emph{On the lifespan of solutions of nonlinear wave equations
  with small initial data}, Comm. Pure Appl. Math. \textbf{43} (1990), no.~4,
  445--472.

\bibitem{Lind92}
\bysame, \emph{Global solutions of nonlinear wave equations}, Comm. Pure Appl.
  Math. \textbf{45} (1992), no.~9, 1063--1096.

\bibitem{LR03}
H.~Lindblad and I.~Rodnianski, \emph{The weak null condition for {E}instein's
  equations}, C. R. Math. Acad. Sci. Paris \textbf{336} (2003), no.~11,
  901--906.

\bibitem{LS05_2}
H.~Lindblad and A.~Soffer, \emph{A remark on asymptotic completeness for the
  critical nonlinear {K}lein-{G}ordon equation}, Lett. Math. Phys. \textbf{73}
  (2005), no.~3, 249--258.

\bibitem{LS05_1}
\bysame, \emph{A remark on long range scattering for the nonlinear
  {K}lein-{G}ordon equation}, J. Hyperbolic Differ. Equ. \textbf{2} (2005),
  no.~1, 77--89.

\bibitem{LS06}
\bysame, \emph{Scattering and small data completeness for the critical
  nonlinear {S}chr\"odinger equation}, Nonlinearity \textbf{19} (2006), no.~2,
  345--353.

\bibitem{LS15}
\bysame, \emph{Scattering for the {K}lein-{G}ordon equation with quadratic and
  variable coefficient cubic nonlinearities}, Trans. Amer. Math. Soc.
  \textbf{367} (2015), no.~12, 8861--8909.

\bibitem{Rauch78}
J.~Rauch, \emph{Local decay of scattering solutions to {S}chr\"{o}dinger's
  equation}, Comm. Math. Phys. \textbf{61} (1978), no.~2, 149--168.

\bibitem{Schl07}
W.~Schlag, \emph{Dispersive estimates for {S}chr\"{o}dinger operators: a
  survey}, Mathematical aspects of nonlinear dispersive equations, Ann. of
  Math. Stud., vol. 163, Princeton Univ. Press, Princeton, NJ, 2007,
  pp.~255--285.

\bibitem{Sh85}
J.~Shatah, \emph{Normal forms and quadratic nonlinear {K}lein-{G}ordon
  equations}, Comm. Pure Appl. Math. \textbf{38} (1985), no.~5, 685--696.

\bibitem{S06}
A.~Soffer, \emph{Soliton dynamics and scattering}, International {C}ongress of
  {M}athematicians. {V}ol. {III}, Eur. Math. Soc., Z\"{u}rich, 2006,
  pp.~459--471.

\bibitem{Sterb16}
J.~Sterbenz, \emph{Dispersive decay for the 1{D} {K}lein-{G}ordon equation with
  variable coefficient nonlinearities}, Trans. Amer. Math. Soc. \textbf{368}
  (2016), no.~3, 2081--2113.

\bibitem{Stingo18}
A.~Stingo, \emph{Global existence and asymptotics for quasi-linear
  one-dimensional {K}lein-{G}ordon equations with mildly decaying {C}auchy
  data}, Bull. Soc. Math. France \textbf{146} (2018), no.~1, 155--213.

\end{thebibliography}

\end{document}